\documentclass{amsart}

\usepackage[utf8]{inputenc}
\usepackage{amsmath,amsthm}
\usepackage{amsfonts}
\usepackage{amssymb}

\usepackage{enumerate}  
\usepackage{hyperref}  
\hypersetup{colorlinks=true,linktoc=all,linkcolor=blue,citecolor=blue}

\usepackage{lmodern}
\usepackage[T1]{fontenc}

\theoremstyle{plain}
\newtheorem{theorem}{Theorem}[section]  
\newtheorem{proposition}[theorem]{Proposition}
\newtheorem{lemma}[theorem]{Lemma}
\theoremstyle{definition}
\newtheorem{definition}[theorem]{Definition}
\theoremstyle{remark}
\newtheorem{remark}[theorem]{Remark}

\newcommand{\R}{\mathbb{R}}  
\newcommand{\di}{\mathrm{d}}  
\newcommand{\ep}{\varepsilon}  
\newcommand{\C}{\mathcal{C}}  
\newcommand{\D}{\mathcal{D}}  
\newcommand{\J}{\mathcal{J}_{\pm \ep}}  
\newcommand{\A}{\mathcal{A}}
\newcommand{\M}{\mathcal{M}}  
\newcommand{\Z}{\mathcal{Z}}  

\DeclareMathOperator{\As}{A_{\Omega}^{\textit{s}}}
\DeclareMathOperator{\divr}{div}  
\DeclareMathOperator{\minimum}{min}  
\DeclareMathOperator{\maximum}{max}  
\DeclareMathOperator{\dist}{dist}  
\DeclareMathOperator{\tr}{tr}  

\numberwithin{equation}{section}

\begin{document}

\title[Bubbling solutions for nonlocal elliptic problems]{Bubbling solutions for nonlocal elliptic problems}

\author{Juan D\'avila} 
\address{\noindent J. D\'avila -
Departamento de Ingenier\'ia Matem\'atica and CMM, Universidad
de Chile, Casilla 170 Correo 3, Santiago, Chile.}
\email{jdavila@dim.uchile.cl}

\author{Luis L\'{o}pez R\'ios}
\address{\noindent L. L\'opez R\'ios -
Departamento de Matem\'atica, Universidad de Buenos Aires, Ciudad Universitaria - Pabell\'on I - (C1428EGA) - Buenos Aires, Argentina}
\email{llopez@dm.uba.ar}

\author{Yannick Sire} 
\address{\noindent Y. Sire -
Institut de Math\'ematiques de Marseille,
Aix-Marseille Universit\'e,
9, rue F. Joliot Curie,
13453 Marseille Cedex 13 FRANCE}
\email{sire@cmi.univ-mrs.fr}
\date{}

\begin{abstract}

We investigate bubbling solutions for the nonlocal equation 
\[
  \As u =u^p,\ u >0 \quad \mbox{in } \Omega,		
\]
under homogeneous Dirichlet conditions, where $\Omega$ is a bounded and smooth domain. The operator $\As$ stands for two types of nonlocal operators that we treat in a unified way: either the spectral fractional Laplacian or the restricted fractional Laplacian. In both cases $s \in (0,1)$ and the Dirichlet conditions are different: for the spectral fractional Laplacian, we prescribe $u=0$ on $\partial \Omega$ and for the restricted fractional Laplacian, we prescribe $u=0$ on $\mathbb R^n \backslash \Omega$. We construct solutions when the exponent $p = (n+2s)/(n-2s) \pm \ep$ is close to the {\em critical one}, concentrating as $\ep \to 0$ near critical points of a reduced function involving the Green and Robin functions of the domain.

\end{abstract}

\maketitle

\tableofcontents


\section{Introduction}

This paper studies the existence of bubbling solutions for the problem  

\begin{equation} \label{1.1}
 \left\{
  \begin{aligned}
    \As u & = u^p , \ 
   u>0 \quad \text{in } \Omega,\\
   & u = 0 \quad \text{on } \Sigma,
  \end{aligned}
 \right.
\end{equation}
where $\Omega$ is a smooth bounded domain in $\R^n$, $s \in (0,1)$, $n > 2s$, $p = (n+2s)/(n-2s) \pm \ep$ ($\ep > 0$ small) and $\As$ is an operator of fractional order with suitable boundary conditions on $\Sigma$ (see below). 

For the usual Laplacian $s=1$, problem \eqref{1.1} was extensively studied when the exponent $p$ approaches the critical one from below, namely $p = (n+2s)/(n-2s) - \ep$, see Br\'ezis and Peletier \cite{BrPe1989}, Rey \cite{Re1989, Re1990, Re1999}, Han \cite{Ha1991} and Bahri, Li and Rey \cite{BaLiRe1995}. In the latter reference, bubbling solutions are found for $n \ge 4$, concentrating around nondegenerate critical points of certain object which involve the Green's and Robin's function of $\Omega$. On the other hand, the supercritical case $p = (n+2s)/(n-2s) + \ep$ was studied by del Pino, Felmer and Musso \cite{dPFeMu2002,dPFeMu2003}, in particular they showed a concentration phenomena for bubbling solutions to this problem when the domain satisfies certain "topological conditions''; for instance a domain exhibiting multiple holes.   

The purpose of the present work is to develop such a theory for equations involving fractional order operators. In the last decade, several works have been devoted to equations involving fractional operators. The canonical example is the so-called fractional laplacian $(-\Delta)^s$, $s \in (0,1)$ in $\R^n$, the Fourier multiplier of symbol $|\xi|^{2s}$. In this respect, the semi-linear equation 
$$(-\Delta)^s\, u = f(u)\,\,\,\mbox{in}\,\,\,\R^n,$$
for a certain function $f:\R^n \rightarrow \R$ attracted a lot of attention (see for instance and references therein \cite{CaTa2010,Ta2011,CS1,CS2,SV,DS,BrCodPSa2013}). 

Coming back to the problem under consideration, for the subcritical case, Choi, Kim and Lee \cite{ChKiLe2014} developed a nonlocal analog of the results by Han \cite{Ha1991} and Rey \cite{Re1990} above mentioned. They also proved Theorem~\ref{bubbling_solutions1} below in the case $p = p^*-\ep$. With a new framework in the spirit of \cite{dPFeMu2002,dPFeMu2003}, we will be able to generalize the work by Choi, Kim and Lee, and consider both the subcritical and supercritical case.

Furthermore, we treat in a unified way two types of operators, denoted here $\As$: the spectral fractional Laplacian and the restricted fractional Laplacian. We now describe them more thoroughly. 

\subsubsection*{The spectral Laplacian} Consider the classical Dirichlet Laplacian $\Delta_{\Omega}$ on the domain $\Omega$\,; then the {\em spectral definition} of the fractional powers of $\Delta_{\Omega}$ relies on the following formula in terms of the semigroup associated to the Laplacian, namely
\begin{equation}\label{sLapl.Omega.Spectral}
\displaystyle(-\Delta_{\Omega})^{s}\,
u(x)=\sum_{j=1}^{\infty}\lambda_j^s\, \hat{u}_j\, \phi_j(x)
=\frac1{\Gamma(-s)}\int_0^\infty
\left(e^{t\Delta_{\Omega}}u(x)-u(x)\right)\frac{\di t}{t^{1+s}}.
\end{equation}
where $\Gamma$ is the gamma function and $\lambda_j>0$, $j=1,2,\dotsc$ are the eigenvalues of the Dirichlet Laplacian on $\Omega$\,, written in increasing order and repeated according to their multiplicity and $\phi_j$ are the corresponding normalized eigenfunctions, namely
\[
\hat{u}_j=\int_\Omega u(x)\phi_j(x)\,dx\,,\qquad\mbox{with}\qquad \|\phi_j\|_{L^2(\Omega)}=1\,.
\]
We will denote the operator defined in such a way as $\As = (-\Delta_{\Omega})^s$, and call it the \textit{spectral fractional Laplacian}. 

\subsubsection*{The restricted fractional Laplacian} For $s \in (0,1)$, one can define a fractional Laplacian operator by using the integral representation in terms of hypersingular kernels
\begin{equation}\label{sLapl.Rd.Kernel}
(-\Delta)^{s}\, u(x)= c_{n,s}\mbox{
P.V.}\int_{\mathbb{R}^n} \frac{u(x)-u(z)}{|x-z|^{n+2s}}\,dz,
\end{equation}
where $c_{n,s}>0$ is a normalization constant. One can ``restrict'' the operator to functions that are zero outside $\Omega$: we will denote the operator defined in such a way as $\As = (-\Delta_{|\Omega})^s$, and call it \textit{the restricted fractional Laplacian}. In this case the operator $(-\Delta_{|\Omega})^s$ is a selfadjoint operator on $L^2(\Omega)$, with a discrete spectrum: we will denote by $\lambda_{s, j}>0$, $j=1,2,\dotsc$ its eigenvalues written in increasing order and repeated according to their multiplicity and we will denote by $\phi_{s, j}$ the corresponding normalized eigenfunctions.

\vspace{10pt}

In view of the previous discussion, the boundary conditions in \eqref{1.1} have to be interpreted in the following way: 
\begin{itemize}
\item $u=0$ on $\Sigma = \partial \Omega$ for the spectral fractional Laplacian,
\item $u=0$ on $\Sigma = \R^n \backslash \Omega$ for the restricted fractional Laplacian.
\end{itemize}

In the entire space, the operator in \eqref{sLapl.Rd.Kernel}, denoted $(-\Delta)^s$, can be defined through Fourier transform $\mathcal F$, by 
\[
  \mathcal F [ ( -\Delta)^s u ] (\zeta) = |\zeta|^{2s} \mathcal F [u](\zeta). 
\]
Throughout the paper, $p^*:=(n+2s)/(n-2s)$ represents the critical Sobolev exponent. For this exponent, the corresponding equation in $\R^n$
\begin{equation}\label{1.2}
 (-\Delta)^{s}\, u = u^{p^*}
\end{equation}
has an explicit family of solutions of the form
\[
  w_{\lambda,\xi}(x) = \lambda^{-\frac{n-2s}{2}} w(\lambda^{-1}(x-\xi))
\]
with $\xi \in \R^n$ and $\lambda>0$, where
\[
 w(x) = \frac{b_{n,s} }{(1+|x|^2)^{\frac{n-2s}{2}}}
\]
and $b_{n,s}$ is a positive constant (see \cite{chLiOu2006} for classification results).

We construct solutions of \eqref{1.1} that concentrate at certain points in $\Omega$ as $\ep \to 0$. These concentration points are determined by the critical points of a map which involves the Green's function of the operator $\As$ and its regular part. Let $G$ denote the Green's function for $\As$ in $\Omega$, that is, for any $\xi \in \Omega$, $G(\cdot,\xi)$ satisfies
\begin{equation} \label{greenfun}
  \begin{gathered}
    \As G(\cdot,\xi) = \delta_\xi (\cdot) \quad \text{in } \Omega, \\
    G(\cdot,\xi) = 0 \quad \text{on } \Sigma,
  \end{gathered}
\end{equation}
where $\delta_\xi$ denotes the Dirac mass at the point $\xi$. In the entire space, we denote the Green function by $\varGamma$, which satisfies
\begin{gather*}
  (-\Delta)^s\, \varGamma(x,\xi) = \delta_\xi(x) \quad\text{for all } x \in \R^n, \\
  \lim_{|x|\to\infty} \varGamma(x,\xi) = 0,
\end{gather*}
for each fixed $\xi \in \R^n$.
The function $\varGamma$ is explicitly given by
\begin{equation} \label{Gamma}
\varGamma(x,\xi) = \frac{a_{n,s}}{|x-\xi|^{n-2s}},
\end{equation}
where $a_{n,s}$ is a positive constant. We also define the regular part of the Green function $G$ of $\Omega$ by
\begin{equation}\label{regularpart}
  H(x,\xi) = \varGamma(x,\xi) - G(x,\xi) \quad \text{for } x,\xi \in \Omega,\ x \not= \xi.
\end{equation}

Given $m \in \mathbb{N}$, the following function will prove to be very important for constructing solutions of \eqref{1.1}:
\begin{equation} \label{Psi}
\Psi(\xi, \Lambda)= \frac{1}{2} \left\{\sum_{i=1}^m H(\xi_i,\xi_i) \Lambda_i^2-2\sum_{i<j} G(\xi_i,\xi_j)\Lambda_i \Lambda_j \right\} \pm \log(\Lambda_1 \dotsm \Lambda_m),
\end{equation}
$\xi = (\xi_1,\dotsc,\xi_m) \in \Omega^m$ and $\Lambda = (\Lambda_1,\dotsc,\Lambda_m) \in (0,\infty)^m$ (see \eqref{psi}).

We recall the definition of stable critical set introduced by Y.Y. Li \cite{Li1997}.

\begin{definition}[Stable critical set] \label{stable_critical_set}
Let $\A$ be a bounded set of critical points of $\Psi$. We say that $\A$ is a stable critical set if for all $\mu>0$ there is a number $\delta>0$ such that if $\Phi \in C^1(\Omega)$ and
\[
  \max_{\dist (\xi,\A) \le \mu} (|\Psi(\xi)-\Phi(\xi)| + |\nabla \Psi(\xi) - \nabla \Phi(\xi)|) < \delta,
\]
then $\Phi$ has at least one critical point $\xi$, with $\dist (\xi,\A) < \mu$.
\end{definition}

We will now state the main results of this paper. Let us start with a concentration result of {\em multiple bubble solutions}.
\begin{theorem} \label{bubbling_solutions1}
Suppose that $\Psi$ in \eqref{Psi} has a stable critical set $\mathcal{A}$. Then, for every point $(\xi_1,\dotsc,\xi_m,\Lambda_1,\dotsc,\Lambda_m) \in \A$ there exists a family of solutions of problem \eqref{1.1} which blow up and concentrate at each point $\xi_i$, $i=1,\dotsc,m$, as $\ep$ tends to zero.
\end{theorem}

Actually, the proof will provide much finer information on the asymptotic profile of the blow up of these solutions as $\ep \to 0$. Up to a scaling and translation, the solutions look around each $\xi_i$ like a bubble, which is a solution in the entire $\R^n$ of the equation at the critical exponent. More precisely, we will find 
\begin{equation} \label{asymptotic_behavior}
  u_\ep(x) = b_{n,s} \sum_{i=1}^{m} \left(  \frac{ \beta_{n,s}^{1/2} \Lambda_{i\ep}^2 \ep^{1-\frac{1}{2s \pm \ep (n-2s/2)}}}{\left[ (\beta_{n,s} \Lambda_{i\ep} \ep)^{2/n-2s} + |x-\xi_{i\ep}|^2 \right]^{(n-2s)/2}}  \right) + \theta_\ep(x),
\end{equation}
where $\theta_\ep(x) \to 0$ uniformly as $\ep \to 0$, $\xi_{i\ep} \to \xi_i$, and $\Lambda_{i\ep} \to \Lambda_i$ up to subsequences. The positive constant $\beta_{n,s}$ will be defined in Section~3.

There is not a general method to find stable critical points of $\Psi$ in \eqref{Psi}. However, in some special domains depending on the criticality of the exponent $p$, we shall show how to find some of these points and then prove the existence of a concentrating family of solutions to \eqref{1.1}. 

\subsubsection*{The supercritical case, two-bubble solutions} In the supercritical case, i.e. $p = p^* + \ep$, if we look for two-bubble solutions ($m = 2$) to \eqref{1.1}, the criticality with respect to $\Lambda$ in \eqref{Psi} can be reduced and the following function will play a crucial role:
\begin{equation} \label{varphi_function}
\varphi(\xi_1,\xi_2) = H^{1/2}(\xi_1,\xi_1)H^{1/2}(\xi_2,\xi_2)-G(\xi_1,\xi_2).
\end{equation} 
A min-max argument shall be used to find suitable critical points of the previous function. We will show then a blowing up and concentration phenomenon at exactly two points $\xi_1,\xi_2$, as $\ep \to 0$, provided that the set where $\varphi < 0$ is ``topologically nontrivial'' in a sense specified below. The pair $(\xi_1,\xi_2)$ will be a critical point of $\varphi$ with $\varphi(\xi_1,\xi_2)<0$. 

Given $B \subset \Omega$, we will denote by $H^d(B)$ its $d$-th cohomology group with integral coefficients and by $\iota^*$ the homomorphism $\iota^* : H^d(\Omega) \to H^*(B)$, induced by the inclusion $\iota : B \to \Omega$.
 
\begin{theorem}\label{two-bubble_solutions}
Consider the supercritical case in problem \eqref{1.1}, i.e. $p = p^* + \ep$. Assume that $\Omega$ is a smooth bounded domain in $\R^n$ that satisfies the following property: There exist a compact manifold $\M \subset \Omega$ and an integer $d \geq 1$ such that, $\varphi < 0$ on $\M \times \M$, $\iota^* : H^d(\Omega) \to H^d(\M)$ is nontrivial and either $d$ is odd or $H^{2d}(\Omega) = 0$. Then there exists $\ep_0 > 0$ such that, for any $0 < \ep <\ep_0$, problem \eqref{1.1} has at least one solution $u_\ep$. Moreover, let $\mathcal{N}$ be the component of the set where $\varphi < 0$ which contains $\M \times \M$. Then, given any sequence $\ep = \ep_j \to 0$, there is a subsequence, which we denote in the same way, and a critical point $(\xi_1,\xi_2) \in \mathcal{N}$ of the function $\varphi$ such that $u_\ep \to 0$ on compact subsets of $\Omega \backslash \{\xi_1,\xi_2\}$ and such that for any $\delta > 0$
\[
  \sup_{|x-\xi_i| < \delta} u_\ep (x) \to \infty, \quad i=1,2,
\]
as $\ep \to 0$.
\end{theorem}

The asymptotic profile of the blow up of the solutions is like \eqref{asymptotic_behavior}, but this time one can identify the limits as
\[
\Lambda_1^2 = -\frac{H(\xi_2,\xi_2)^{1/2}}{H(\xi_1,\xi_1)^{1/2}\varphi(\xi_1,\xi_2)}, \quad \Lambda_2^2 = -\frac{H(\xi_1,\xi_1)^{1/2}}{H(\xi_2,\xi_2)^{1/2}\varphi(\xi_1,\xi_2)}.
\]

To clarify the meaning of Theorem~\ref{two-bubble_solutions}, we mention two examples under the scope of this result. The first one is a domain $\D$ with an excised subdomain $\omega$ contained in a ball of sufficiently small radius. The second example is a domain $\D \subset \R^3$ from which one takes away a solid torus with sufficiently small cross-section. For more details we refer the reader to \cite{dPFeMu2003}. 

\subsubsection*{The subcritical case, one-bubble solutions} In the subcritical case, i.e. $p = p^* - \ep$, if we look for one-bubble solutions ($m=1$) to \eqref{1.1}, the function $\Psi$ in \eqref{Psi} takes the simple form
\[
  \Psi(\xi, \Lambda)= \frac{1}{2} H(\xi,\xi) \Lambda^2 - \log \Lambda, \quad \xi \in \Omega, \Lambda>0.
\]
$H(\xi,\xi)$ is called the Robin's function of $\Omega$. In Section~6 we will show that 
\begin{equation*}
  c_1 d(\xi)^{2s-n} \le H(\xi,\xi) \le c_2 d(\xi)^{2s-n} \quad \text{for all } \xi \in \Omega,
\end{equation*}
where $d(\xi) := \dist(\xi, \partial \Omega)$ and $c_1,c_2>0$, see Lemma~\ref{robin_function_lemma}.
Therefore $H(\xi,\xi)$ blows up at the boundary, which implies that its absolute minima are stable under small variations of it.  

\begin{theorem}\label{theorem2}
Consider the subcritical case in problem \eqref{1.1}, that is $p = p^*-\ep$. Then, there exists a family of solutions which blow up and concentrate, as $\ep$ tends to zero, at an absolute minimum of the Robin's function of $\Omega$.
\end{theorem}

The paper is organized as follows. 
In order to keep it easy to read, we have chosen to concentrate in Sections 2--6 on the results dealing with the {\em spectral fractional Laplacian}. In Section 7, the details which have to be changed in the theory from the spectral fractional Laplacian to the {\em restricted fractional Laplacian} will be explained. We refer the reader to the paper \cite{BSV} where a thorough analysis of the differences between the spectral fractional Laplacian and the restricted fractional Laplacian is performed; in particular, as far as their domains are concerned and several other properties of their eigen-elements. In Section~2 we recall the definition and basic properties of the fractional Laplacian in bounded domains and in the whole $\R^n$. In Section~3 we shall develop the analytical tools toward the main results. We study the linearization around special entire solutions of \eqref{1.2}; an initial approximation will be done as well. Section~4 and 5 contain the reduction to a finite dimensional functional and its relation with the original problem \eqref{1.1}; these sections contain the final tools to prove Theorem~\ref{bubbling_solutions1}--\ref{theorem2} in Section~6. Finally, in Section~7 we complete the proof of the previous theorems by studying the corresponding properties for the {\em restricted fractional Laplacian}.


\section{Preliminary results }

In this section we recall some basic properties of the spectral fractional Laplacian. The notations used throughout this paper are settled down as well.

On a smooth bounded domain $\Omega$, we consider
\[
  (-\Delta_\Omega)^s = \sum_{i=1}^\infty \lambda_i^s P_i
\]
where $\{\lambda_i,\phi_i\}_{i=1}^\infty$ are the eigenvalues and corresponding eigenvectors of $-\Delta_\Omega$ on $H_0^1(\Omega)$ and $P_i$ is the orthogonal projection on the eigenspace corresponding to $\lambda_i$. Denote 
$$
H(\Omega)=\{u=\sum_{i=1}^{\infty} a_i\phi_i \in L^2(\Omega): \sum_{i=1}^{\infty} a_i^2\lambda_i^{s} < \infty\}.
$$
The operator $(-\Delta_\Omega)^s$ is an isomorphism between $H(\Omega)$ and its dual. This space can be characterized more explicitly, see \cite{BSV,CaDaDuSi2011}.

As it is now well-known, the Caffarelli-Silvestre extension \cite{CaSi2007} provides a powerful tool to handle problems (and do computations) involving nonlocal operators modeled on the fractional laplacian, which is our case here. We now describe this extension in our context (see \cite{CaTa2010,Ta2011,CaDaDuSi2011}). These two description are actually equivalent once a suitable functional setting is defined.

The extension problem is set in the cylinder $\Omega\times (0,\infty)$ and it will be convenient to use the following notation: $x \in \R^n$, $y>0$, and $X = (x,y) \in \R^n_+ := \R^n \times (0,\infty)$; likewise, we denote by $\C$ the cylinder $\Omega \times (0,\infty)$ and by $\partial_L \C$ its lateral boundary, i.e. $\partial \Omega \times (0,\infty)$. The ambient space $H_{0,L}^s(\C)$ is defined as the completion of
\[
  C_{0,L}^s(\C):=\{U\in C^\infty(\overline{\C}): U=0 \text{ on }  \partial_L \C\}
\] 
with respect to the norm
\begin{equation} \label{1.3}
  \|U\|_\C=\left(\int_{\C}y^{1-2s}|\nabla U|^2\right)^{1/2}.
\end{equation}
This is a Hilbert space endowed with the following inner product
\[
  \left< U,V \right>=\int_{\C}y^{1-2s}\nabla U \cdot \nabla V \quad \text{for all } U,V \in H_{0,L}^s(\C).
\]
In the entire space, we denote by $\D^s(\R^{n+1}_+)$ the completion of $C_0^{\infty}(\overline{\R^{n+1}_+})$ with respect to the norm $\|\cdot\|_{\R^{n+1}_+}$ defined as in \eqref{1.3}. We point out that if $\Omega$ is a smooth bounded domain then
\[
  H(\Omega)=\{u=\tr |_{\Omega\times \{0\}}U: U\in H_{0,L}^s(\C)\}.
\]
The extension problem is the following: given $u \in H(\Omega)$, we solve
\begin{equation} \label{extension}
  \left\{
    \begin{aligned}
       \divr( y^{1-2s} & \nabla U ) = 0 && \text{in } \C, \\
       U &= 0 && \text{on } \partial_L \C, \\
       U &= u && \text{on } \Omega,
    \end{aligned}
  \right.
\end{equation}
for $U\in H_{0,L}^s(\C)$, where divergence and $\nabla$ are operators acting on all variables $X=(x,y)$.
Then, up to a multiplicative constant,
\begin{equation} \label{extension-b}
(-\Delta_\Omega)^s u  = - \lim_{y \to 0} y^{1-2s} \partial_y U,
\end{equation}
where $c = c(n,s)>0$ (see \cite{CaSi2007} and \cite{CaDaDuSi2011} for the entire and bounded domain case, respectively).

Regarding this extension procedure, the Green function defined in \eqref{greenfun} can be seen, up to a positive constant, as the trace of the solution $G$ for the following extended Dirichlet-Neumann problem
\begin{equation}\label{greens_function}
  \left\{
    \begin{aligned}
       \divr( y^{1-2s} \nabla  G(\cdot,\xi)) &= 0 && \text{in } \C, \\
       G(\cdot,\xi) &= 0 && \text{on } \partial_L \C, \\
       -\lim_{y\to0} y^{1-2s} \partial_y G(\cdot,\xi) &= \delta_\xi(\cdot) && \text{on } \Omega,
    \end{aligned}
  \right.
\end{equation}
$\xi \in \Omega$ (we denote the Green function, as well as its extension, by $G$). Moreover, we have the following representation formula
\begin{equation} \label{greenformula}
U(z)=\int_{\Omega}G(z,\xi)(-\Delta_\Omega)^su(\xi) \, \di \xi \quad \text{for all } z \in \C,
\end{equation}
where $u=\tr |_{\Omega \times \{0\}}U$. Likewise, the regular part of the Green function defined in \eqref{regularpart} can be extended in $H_{0,L}^s(\C)$ as the unique solution of
\begin{equation}\label{regular-extension}
  \left\{
    \begin{aligned}
       & \divr( y^{1-2s} \nabla  H(z,\xi)) = 0, && z \in \C, \\
       & H(z,\xi) = \varGamma(z-\xi), && z \in \partial_L \C, \\
       & \lim_{y\to0} y^{1-2s} \partial_y H(z,\xi) = 0, && z \in \Omega,
    \end{aligned}
  \right.
\end{equation}
$\xi \in \Omega$ (we denote the regular part of the Green function, as well as its extension, by $H$).

In the next sections, given a a function $u \in H(\Omega)$, when we speak of its $s$-harmonic extension to $\Omega\times(0,\infty)$ we will always refer to the solution of \eqref{extension}. This extension process depends on the domain, and we include the possibility that the domain is $\R^n$, in which case $U$ can be written as a convolution of $u$ and an explicit kernel
\begin{equation}\label{conv}
U(x,y) = \int_{\R^n} P(x-t,y) u(t) \, \di t
\end{equation}
where 
$$
P(x,y) = C_{n,s} \frac{y^{2s}}{ (|x|^2 +y^2)^{\frac{n+2s}{2}}}
$$
(see \cite{CaSi2007}). Then, the $s$-harmonic extension of the fundamental solution \eqref{Gamma} to $\R^n_+ := \R^n \times (0,\infty)$ is given just by 
\[
  \varGamma(z_1,z_2) = \frac{a_{n,s}}{|z_1-z_2|^{n-2s}} \quad\text{for } z_1,z_2 \in \R^n_+, \ z_1\not= z_2.
\]

We end this section with the folllowing maximum principle.
\begin{lemma}[Maximum principle] \label{maximum_principle}
Suppose that $U$ is a weak solution of the problem
\begin{equation*}
  \left\{
    \begin{aligned}
       \divr( y^{1-2s} \nabla  U) &= 0 && \text{in } \C, \\
       U &= g && \text{on } \partial_L \Omega, \\
       \lim_{y \to 0} y^{1-2s} \partial_y U &= 0 && \text{on } \Omega,
    \end{aligned}
  \right.
\end{equation*}
for some function $g:\partial_L \Omega \rightarrow \R$. Then
\[
  \sup_{z \in \C} |U(z)| \leq \sup_{z \in \partial_L \C}|g(z)|.
\]
\end{lemma}

\begin{proof}
Let $\bar U = \sup_{z \in \partial_L \C} g(z)$, and consider the function $V(z) = \bar U - U(z)$ which satisfies
\begin{equation*}
  \left\{
    \begin{aligned}
       \divr( y^{1-2s} \nabla  V) &= 0 && \text{in } \C, \\
       V & \geq 0  &&                   \text{on } \partial_L \Omega, \\
       \lim_{y \to 0} y^{1-2s} \partial_y V &= 0 && \text{on } \Omega.
    \end{aligned}
  \right.
\end{equation*}
Note that $V^+=0$ on $\partial_L \Omega$. Then, we deduce that
\[
 0 = \int_{\C} y^{1-2s} \nabla V \cdot \nabla V^+ = -\int_{\C} y^{1-2s} |\nabla V^+|^2. 
\] 
It implies that $V^+ = 0$, and then $U \leq \bar U$ in $\C$. By a similar argument, we can deduce that $\inf_{z \in \partial_L \C} g(z) \leq U$ in $\C$, which completes the proof.
\end{proof}


\section{Initial approximation and reduced energy}

Let $\Omega$ be a bounded domain with smooth boundary in $\R^n$. It will be convenient to work with the enlarged domain 
\[
  \Omega_\ep = \ep ^{-\frac{1}{n-2s}}\Omega,
\] 
$\ep > 0$ small, that, after the change of variables 
\[
  v(x) = \ep ^{\frac{1}{2s \pm \ep (n-2s/2)}}u(\ep ^{\frac{1}{n-2s}}x), \quad x \in \Omega_\ep,
\]
transforms equation \eqref{1.1} into
\begin{equation}\label{main_equation}
  \left\{
    \begin{aligned}
      (-\Delta_{\Omega_\ep})^s v & = v^{p^* \pm \ep}, \  v >  0 \quad \text{in }\Omega_\ep, \\ 
      v & = 0 \quad \text{on } \partial \Omega_\ep
  \end{aligned}
\right.
\end{equation}
(recall that $p^*:=(n+2s)/(n-2s)$).

As $\ep > 0$ is small, we shall develop an initial approximation based on solutions of the equation
\begin{equation} \label{entire-solutions}
 (-\Delta)^{s}\, v = v^{p^*} \quad \text{in } \R^n.
\end{equation}
Specifically, the family generated by
\[
 w(x) = \frac{b_{n,s} }{(1+|x|^2)^{\frac{n-2s}{2}}}
\]
in the following way:
\begin{equation} \label{w_definition}
  w_{\lambda,\xi}(x) = \lambda^{-\frac{n-2s}{2}} w(\lambda^{-1}(x-\xi)) = b_{n,s} \left( \frac{\lambda}{\lambda^2 + |x-\xi|^2} \right)^{\frac{n-2s}{2}},
\end{equation}
with $\lambda>0$ and $\xi \in \R^n$. Here $b_{n,s}$ is a positive constant (see \cite{chLiOu2006} for classification results). 

Let $W_{\lambda,\xi}$ denote the $s$-harmonic extension of $w_{\lambda,\xi}$ to $\R^{n+1}_+$ given by the formula \eqref{conv}, so that $W_{\lambda,\xi}$ satisfies
\begin{equation} \label{w-extension}
  \left\{
    \begin{gathered}
    \divr(y^{1-2s} \nabla  W_{\lambda,\xi} ) = 0 \quad \text{in } \R^{n+1}_+, \\
    W_{\lambda,\xi} = w_{\lambda,\xi} \quad \text{on } \R^n.
    \end{gathered}
  \right.
\end{equation}
To deal with the zero Dirichlet condition in \eqref{1.1}, we introduce the function $v_{\lambda,\xi}$ to be the $H(\Omega_\ep)$-projection of $w_{\lambda,\xi}$, namely the unique solution of the equation 
\begin{equation}\label{v_i}
  \left\{
    \begin{gathered}
      (-\Delta_{\Omega_\ep} )^s v_{\lambda,\xi} = w_{\lambda,\xi}^{p^*} \quad \text{in } \Omega_\ep, \\
      v_{\lambda,\xi} = 0 \quad \text{on } \partial\Omega_\ep.
    \end{gathered}
  \right.
\end{equation}
The functions $v_{\lambda,\xi}$ can be expressed as
\[
  v_{\lambda,\xi} = w_{\lambda,\xi} - \varphi_{\lambda,\xi} \quad \text{in } \Omega_\ep,
\]
where $\varphi_{\lambda,\xi}$ is the trace on $\Omega_\ep$ of the unique solution $\varPhi_{\lambda,\xi}$ of 
\begin{equation} \label{phi-extension}
  \left\{
    \begin{aligned}
      \divr(y^{1-2s}\nabla  \varPhi_{\lambda,\xi})  &= 0 && \text{in } \C_\ep, \\
      \varPhi_{\lambda,\xi} &= W_{\lambda,\xi} && \text{on } \partial_L \C_\ep, \\
      \lim_{y \to0} y^{1-2s} \partial_y \varPhi_{\lambda,\xi} &= 0 && \text{on } \Omega_\ep
    \end{aligned}
  \right.
\end{equation}
(recall that that $\C_\ep$ is the enlarged cylinder $\Omega_\ep \times (0,\infty)$ and $\partial_L \C_\ep$ its lateral boundary).

We develop an initial approximation with concentration at certain $m$ points $\xi_1,\dotsc,\xi_m\in\Omega$. To this end, we consider the properly scaled points 
\begin{equation}\label{scaled_xi}
\xi_i' = \ep^{-\frac{1}{n-2s}} \xi_i \in \Omega_\ep,
\end{equation}
and, for parameters $\lambda_1,\dotsc,\lambda_m>0$, look for a solution of problem \eqref{main_equation} of the form 
\begin{equation}\label{ansatz}
v = \bar v + \phi,
\end{equation}
where
\[
\bar v = \sum_{i=1}^m v_i, \quad \text{with } v_i = v_{\lambda_i,\xi_i'}.
\]
The points and parameters $\{\xi_i,\lambda_i\}_{i=1}^m$ shall be suitable chosen to made the term $\phi$ of ``small order'' all over $\Omega_\ep$. 

As we pointed out in the previous section (see \eqref{extension} and \eqref{extension-b}), solutions of \eqref{main_equation} are closely related to those of 
\begin{equation}
  \left\{
    \begin{aligned}
       \divr( y^{1-2s} \nabla V ) &= 0 && \text{in } \C_\ep, \\
       V &> 0 && \text{in } \C_\ep, \\
       V &= 0 && \text{on } \partial_L \C_\ep, \\
       - \lim_{y\to0} y^{1-2s} \partial_y V &= v^{p^* \pm \ep} && \text{on } \Omega_\ep.
    \end{aligned}
  \right.
\end{equation}
These functions correspond, in turn, to stationary points of the energy functional
\begin{equation}\label{energy_functional}
J_{\pm \ep}(V) = \frac12 \int_{\C_\ep } y^{1-2s}|\nabla V|^2-\frac{1}{p^*+1 \pm \ep} \int_{\Omega_\ep} |V|^{p^*+1 \pm \ep}.
\end{equation}
We remark that in the subcritical case these functionals are well defined and $C^1$ in the Hilbert space $H^s_{0,L}(\C_\ep)$.

If a solution of the form \eqref{ansatz} exists, we should have $J_{\pm \ep}(V) \sim J_{\pm \ep}(\bar V)$, where $V$ and $\bar V$ denote the s-harmonic extension of $v$ and $\bar v$, respectively. Then the corresponding points $(\xi_1,\dotsc,\xi_m,\lambda_1,\dotsc,\lambda_m)$ in the definition of $\bar v$ are also ``approximately stationary'' for the finite dimensional functional $(\xi_1,\dotsc,\xi_m,\lambda_1,\dotsc,\lambda_m) \mapsto J_{\pm \ep}(\bar V)$. It is then necessary to understand the structure of this functional and find critical points that survive small perturbations. A first approximation is the following: If the points $\xi_i$ are taken far apart from each other and far away from the boundary, 
\[
  J_{\pm \ep}(\bar V) \sim \sum_{i=1}^m J_{\pm \ep}(V_i) \sim m C_{n,s}
\]
where
\[
  C_{n,s} = \frac12 \int_{\R^{n+1}_+} y^{1-2s}|\nabla W|^2-\frac{1}{p+1} \int_{\R^n} |w|^{p^*+1},
\]
and $V_i$ and $W$ are the s-harmonic extension of $v_i$ and $w$, respectively. 

To work out a more precise expansion, it will be convenient to recast the variables $\lambda_i$ into the $\Lambda_i$'s given by
\begin{equation}\label{capital lambda}
\lambda_i = (\beta_{n,s} \Lambda_i)^{\frac{1}{n-2s}}
\end{equation}
with
\[
\beta_{n,s} = \frac{\int_{\R^n}w^{p^*+1}}{(p^*+1)(\int_{\R^n}w^{p^*})^2}
\]
In order to get good estimates of $J_{\pm \ep}(\bar V)$, we take the concentration points uniformly separated in $\Omega$ and stay away from the boundary. Let us  fix a small $\delta>0$ and work with $ \xi_1,\dotsc,\xi_m \in \Omega $, and $ \lambda_1,\dotsc,\lambda_m > 0 $, such that 
\begin{gather}
|\xi_i-\xi_j| \ge \delta \quad \text{for all } i\neq j \quad \text{and} \quad \dist(\xi_i,\partial\Omega) \ge \delta \quad \text{for all } i; \label{xi_separation} \\ 
\Lambda_i \in (\delta,\delta^{-1}) \quad \text{for all } i. \label{lambda_separation}
\end{gather}
In order to find and expansion of $J_{\pm \ep}(\bar V)$, let us find before one for $\varphi_{\lambda,\xi'}$ and $v_{\lambda,\xi'}$. 

\begin{lemma}
Given $\xi \in \Omega$ and $\lambda > 0$, we have that
\begin{equation} \label{3.1}
  \varphi_{\lambda,\xi'}(\ep^{-\frac{1}{n-2s}}x) = \alpha \lambda^{\frac{n-2s}2} H(x,\xi) \ep + o(\ep), 
\end{equation}
uniformly for $x \in \Omega$. And, away from $x=\xi$,
\begin{equation} \label{3.2}
  v_{\lambda,\xi'}(\ep^{-\frac{1}{n-2s}}x) = \alpha \lambda^{\frac{n-2s}2} G(x,\xi) \ep + o(\ep),
\end{equation}
uniformly for $x$ on each compact subset of $\Omega$. Here $\alpha = \alpha(n,s) = \int_{\R^n} w^{p^*}$ and $G$, $H$ are respectively the Green function of the fractional Laplacian with Dirichlet boundary condition on $\Omega$ and its regular part.  
\end{lemma}

\begin{proof}
Using \eqref{greenformula} and then \eqref{w_definition}, the function $W_{\lambda,\xi'}$ in \eqref{w-extension} can be written as
\begin{align*}
  W_{\lambda,\xi'}(z) &= \int_{\R^n} \varGamma(z,\tau)w_{\lambda,\xi'}^{p^*}(\tau)\, \di \tau \\
                      &= \lambda^{-\frac{n+2s}{2}}\int_{\R^n} \varGamma(z,\tau)w^{p^*}(\lambda^{-1}(\tau-\xi'))\, \di \tau \quad \text{for all } z =(x,y) \in \R^{n+1}_+. 
\end{align*}

Regarding \eqref{regular-extension} and \eqref{phi-extension}, let us now consider the functions $H_\ep(z) = \alpha \lambda^{\frac{n-2s}2} H(x,\xi) \ep$ and $\varPhi_\ep(z) = \varPhi_{\lambda,\xi'}(\ep^{-\frac{1}{n-2s}}z)$, both defined in $\C$. Using the previous identity, we have that,  
\begin{align*}
  \varPhi_\ep(z)= W_{\lambda,\xi'}(\ep^{-\frac{1}{n-2s}}z) &= \lambda^{-\frac{n+2s}{2}}\int_{\R^n} \varGamma(\ep^{-\frac{1}{n-2s}}z,\tau)w^{p^*}(\lambda^{-1}(\tau-\xi'))\, \di \tau \\
   & = \lambda^{\frac{n-2s}{2}}\int_{\R^n} \varGamma(\ep^{-\frac{1}{n-2s}}z,\xi'+\lambda \tau) w^{p^*}(\tau)\, \di \tau \\
   & = \lambda^{\frac{n-2s}{2}} \ep \int_{\R^n} \varGamma(z,\xi+\lambda \ep^{\frac{1}{n-2s}} \tau) w^{p^*}(\tau)\, \di \tau \\
   & = \alpha \lambda^{\frac{n-2s}2} \varGamma(z,\xi) \ep + o(\ep),
\end{align*}
uniformly for $z \in \partial_L \Omega$.
Therefore, 
\[
\sup_{z \in \partial_L \Omega}|\varPhi_\ep(z)-H_\ep(z)|=o(\ep)
\]
By the maximum principle in the previous section, we deduce that
\[
\sup_{z \in \Omega}|\varPhi_\ep(z) - H_\ep(z)|=o(\ep).
\]
This establishes \eqref{3.1}. A similar argument can be used to state \eqref{3.2}.
\end{proof}

\begin{lemma} \label{energy_expansion}
The following expansion holds:
\begin{equation}\label{energy-expansion}
J_{\pm \ep}(\bar V)= m C_{n,s} + [\gamma_{n,s} + \omega_{n,s} \Psi(\xi,\Lambda)]\ep + o(\ep)
\end{equation}
uniformly with respect to $(\xi,\Lambda)$ satisfying \eqref{xi_separation} and \eqref{lambda_separation}. Here
\begin{equation} \label{psi}
  \Psi(\xi, \Lambda)= \frac{1}{2} \left\{\sum_{i=1}^m H(\xi_i,\xi_i) \Lambda_i^2-2\sum_{i<j} G(\xi_i,\xi_j)\Lambda_i \Lambda_j \right\} \pm \log(\Lambda_1 \dotsm \Lambda_m),
\end{equation}
\[
  \gamma_{n,s} = \left\{ \pm \frac{m}{p^*+1}\omega_{n,s} \pm \frac{m}{2}\omega_{n,s} \log \beta_{n,s} \mp \frac{m}{p^*+1} \int_{\R^n} w^{p^*+1} \log w \right\} 
\]
and
\[
  \omega_{n,s} = \frac{\int_{\R^n} w^{p^*+1}}{p^*+1}.
\]
\end{lemma} 
\begin{proof}
Consider the energy functional
\[
J_0(V) = \frac12 \int_{\C_\ep} y^{1-2s}|\nabla V|^2-\frac{1}{p+1} \int_{\Omega_\ep} |V|^{p^*+1}.  
\]
In order to prove \eqref{energy-expansion}, let us first estimate $I_0(\bar V)$. Recall that $\bar v = \sum_{i=1}^m v_i$, and then $\bar V = \sum_{i=1}^m V_i$ where $\bar V$ and $V_i$ represent the $s$-harmonic extension of $\bar v$ and $v_i$, respectively. We have
\begin{equation}\label{3.3}
  \begin{aligned}
    J_0(\bar V) &= J_0(\sum_{i=1}^m V_i) \\
                &= \sum_{i=1}^m \frac12 \int_{\C_\ep} y^{1-2s}|\nabla V_i|^2-\frac{1}{p^*+1} \int_{\Omega_\ep} v_i^{p^*+1}\\
                & + \sum_{i \neq j} \int_{\C_\ep} y^{1-2s} \nabla V_i \nabla V_j - \frac{1}{p^*+1} \left[ \int_{\Omega_\ep} \left( \sum_{i=1}^m v_i \right)^{p^*+1}-\sum_{i=1}^m v_i^{p^*+1} \right].
  \end{aligned}
\end{equation}

Now, recall that, by \eqref{v_i}, $V_i$ satisfies up to a constant
\begin{equation}
  \left\{
    \begin{aligned}
       \divr( y^{1-2s} \nabla V_i ) &= 0 && \text{in } \C_\ep, \\
       V_i &= 0 && \text{on } \partial_L \C_\ep, \\
       - \lim_{y\to0} y^{1-2s} \partial_y V_i &= w_i^{p^*} && \text{on } \Omega_\ep,
    \end{aligned}
  \right.
\end{equation}
where $w_i = w_{\lambda_i,\xi_i'}$. Integrating by parts, we deduce that
\[
  \int_{\C_\ep} y^{1-2s}|\nabla V_i|^2 = \int_{\Omega_\ep} w_i^{p^*} v_i = \int_{\Omega_\ep} w_i^{p^*+1}- w_i^{p^*} \varphi_i.
\]
This and the previous lemma imply
\begin{equation}\label{3.4}
  \begin{aligned}
    \int_{\C_\ep} y^{1-2s}|\nabla V_i|^2 &= \int_{\R^n} w^{p^*+1} - \beta_{n,s} \left( \int_{\R^n} w^{p^*} \right)^2 H(\xi_i,\xi_i) \Lambda_i^2 \ep + o(\ep) \\
    &= \int_{\R^{n+1}_+} y^{1-2s}|\nabla W|^2 - \beta_{n,s} \left( \int_{\R^n} w^{p^*} \right)^2 H(\xi_i,\xi_i) \Lambda_i^2 \ep + o(\ep),
  \end{aligned} 
\end{equation}
where the last equality is due to $W$ is the $s$-harmonic extension of $w$, which satisfies equation \eqref{entire-solutions}.

By a similar argument, we see that
\begin{gather}
  \int_{\C_\ep} y^{1-2s} \nabla V_i \nabla V_j = \beta_{n,s} \left( \int_{\R^n} w^{p^*} \right)^2 G(\xi_i,\xi_j) \Lambda_i \Lambda_j \ep + o(\ep), \label{3.5} \\
  \int_{\Omega_\ep}v_i^{p^*+1} = \int_{\R^n} w^{p^*+1} - (p^*+1) \beta_{n,s} \left( \int_{\R^n} w^{p^*} \right)^2 H(\xi_i, \xi_i) \Lambda_i^2 \ep + o(\ep) \label{3.6}
\end{gather}
and 
\begin{equation} \label{3.7}
  \begin{aligned}
    \frac{1}{p^*+1} &\left[\int_{\Omega_\ep} \left( \sum_{i=1}^m v_i \right)^{p^*+1}-\sum_{i=1}^m v_i^{p^*+1} \right] \\
    & = 2 \beta_{n,s} \left( \int_{\R^n} w^{p^*} \right)^2 G(\xi_i,\xi_j) \Lambda_i \Lambda_j \ep + o(\ep) \quad \text{for all } i \neq j.
  \end{aligned}
\end{equation}
Putting \eqref{3.4}--\eqref{3.7} in \eqref{3.3}, we conclude that
\[
  J_0(\bar V)= m C_{n,s} + \frac{\omega_{n,s}}{2} \left\{\sum_{i=1}^m H(\xi_i,\xi_i) \Lambda_i^2-2\sum_{i<j} G(\xi_i,\xi_j)\Lambda_i \Lambda_j \right\}\ep + o(\ep).
\]

On the other hand, 
\[
J_{\pm \ep}(\bar V) - J_{0}(\bar V) = \pm \frac{\ep}{(p^*+1)^2} \int_{\Omega_\ep} \bar V^{p^*+1} \mp \frac{\ep}{p^*+1}\int_{\Omega_\ep} \bar V^{p^*+1} \log \bar V + o(\ep).
\]
The right hand side can be computed as in \cite[Lemma~2.1]{dPFeMu2002} and \cite{dPFeMu2003}, it gives us the following expansion
\begin{align*}
  J_{\pm \ep}(\bar V) -& J_{0}(\bar V) \\
  =& \left[ \pm \frac{m}{(p^*+1)^2} \int_{\R^n} w^{p^*+1} \pm \frac{m}{2(p^*+1)} \log \beta_{n,s} \left( \int_{\R^n} w^{p^*+1} \right) \right.\\
  & \left. \pm \frac{\int_{\R^n} w^{p^*+1}}{p^*+1}\log (\Lambda_1 \dotsm \Lambda_m) \mp \frac{m}{p^*+1}\int_{\R^n} w^{p^*+1} \log w \right]\ep + o(\ep), 
\end{align*}
which concludes the proof.
\end{proof}

\begin{remark}
The quantity $o(\ep)$ in the expansion above is actually also of that size in the $C^1$-norm as a function of $\xi$ and $\Lambda$ satisfying \eqref{xi_separation} and \eqref{lambda_separation}.
\end{remark}


\section{The finite-dimensional reduction}

In this section we introduce a linear problem in a suitable functional setting which is the basis for the reduction of problem \eqref{1.1} to the study of a finite dimensional problem. Fix a small number $\delta > 0$ and consider points $\xi_i' \in \Omega_\ep$ and numbers $\Lambda_i>0$, $i=1,\dotsc,m$, such that
\begin{equation}\label{constrains}
  \begin{gathered}
    |\xi_i'-\xi_j'| \ge \ep^{-\frac{1}{n-2s}} \delta \quad \text{for all } i\neq j, \\
    \dist(\xi_i',\partial \Omega_\ep) > \ep^{-\frac{1}{n-2s}} \quad \text{and} \quad \delta < \Lambda_i < \delta^{-1} \quad \text{for all } i.
  \end{gathered}
\end{equation}

As we mention in the previous section, we look for solutions to problem \eqref{main_equation} of the form $v=\bar v + \phi$, see \eqref{ansatz}. So we consider the intermediate problem of finding $\phi$ and $c_{ij}$ such that 
\begin{equation} \label{4.1}
  \left\{
    \begin{aligned}
    (-\Delta_{\Omega_\ep})^s(\bar v + \phi) &= (\bar v + \phi)_+^{p^* \pm \ep} + \sum_{i,j} c_{ij} w^{p^*-1} z_{ij} && \text{in } \Omega_\ep, \\
    \phi &= 0  && \text{on } \partial \Omega_\ep, \\
    \int_{\Omega_\ep}\phi w^{p^*-1} z_{ij} &= 0 && \text{for all } i,j,
    \end{aligned}
  \right.
\end{equation}
where $z_{ij}$ are defined as follows: consider the functions
\begin{gather}
  \bar z_{ij} = \frac{\partial w_{\lambda_i,\xi_i'}}{\partial \xi_{ij}'}, \quad 1 \le i \le m,\ 1 \le j \le n, \label{bar-z_a}\\
  \bar z_{i0} = \frac{\partial w_{\lambda_i,\xi_i'}}{\partial \lambda_i} = \frac{n-2s}{2} w_{\lambda_i,\xi_i'}+ (x-\xi_i') \cdot \nabla w_{\lambda_i,\xi_i'}, \quad 1 \le i \le m \label{bar-z_b}
\end{gather} 
and then define the $z_{ij}$'s to their respective $H(\Omega_\ep)$-projection, i.e. the unique solutions of
\begin{equation*}
  \left\{
    \begin{aligned}
    (-\Delta_{\Omega_\ep})^s z_{ij} &= (-\Delta_{\Omega_\ep})^s \bar z_{ij}  && \text{in } \Omega_\ep, \\
    z_{ij} &= 0  && \text{on } \partial \Omega_\ep. \\
    \end{aligned}
  \right.
\end{equation*}

\begin{remark}
\begin{enumerate}[i)]
\item In order to find solutions of \eqref{main_equation}, we have to solve \eqref{4.1} and then find points $\xi_i'$ and scalars $\Lambda_i$ such that the associated $c_{ij}$ are all zero.
\item Observe that for $\phi\in L^\infty(\R^n)$ the integral
\[
\int_{\R^n} \phi w^{p^*-1} z_{ij}
\]
is well defined because $w^{p^*-1}(x)\leq C (1+|x|)^{-4s}$ and 
$|z_{ij}(x)|\leq C (1+|x|)^{-n+2s}$.
\item The role of the functions $\bar z_{ij}$ will be clarified in Proposition~\ref{non-degeneracy}. 
\end{enumerate}
\end{remark}

The first equation of \eqref{4.1} can be rewritten in the following form:
\begin{equation*}
(-\Delta_{\Omega_\ep})^s \phi - (p^* \pm \ep) \bar v^{p^*-1\pm \ep} \phi = R_\ep + N_\ep(\phi) + \sum_{i,j} c_{ij} w^{p^*-1} z_{ij}
\end{equation*}
where
\begin{gather*}
R_\ep = \bar v^{p^* \pm \ep} - \sum_{i=1}^m w_i^{p^*},\\
N_\ep(\phi) = (\bar v + \phi)_+^{p^* \pm \ep} - \bar v^{p^* \pm \ep} - (p^* \pm \ep) \bar v^{p^*-1\pm \ep} \phi.
\end{gather*}
Then we need to understand the following linear problem: given $h \in C^{\alpha}(\bar \Omega_\ep)$, find a function $\phi$ such that for certain constants $c_{ij}$, $i=1,\dotsc,m$, $j=0,\dotsc,n$ one has
\begin{equation}\label{projected_problem}
  \left\{
    \begin{aligned}
    (-\Delta_{\Omega_\ep})^s \phi - (p^* \pm \ep) \bar v^{p^*-1\pm \ep} \phi &= h + \sum_{i,j} c_{ij} w^{p^*-1} z_{ij} && \text{in } \Omega_\ep, \\
    \phi &= 0  && \text{on } \partial \Omega_\ep, \\
    \int_{\Omega_\ep}\phi w^{p^*-1} z_{ij} &= 0 && \text{for all } i,j.
    \end{aligned}
  \right.
\end{equation}
To solve this problem, we consider appropriate weighted $L^\infty$-norms: For a given  $\alpha\geq 0$, let us 
define the following norm of a function $h:\Omega_\ep\to\R$
\[
  \|h\|_{\alpha} = \sup_{x\in \Omega_\ep} 
  \frac{|h(x)|}{\sum_{i=1}^m (1+|x-\xi_i'|)^{-\alpha}} .
\]
With these norms, we have the following a priori estimate for bounded solutions of \eqref{projected_problem}.

\begin{lemma}\label{a_priori_bound}
Let $\alpha>2s$ and assume constrains \eqref{constrains} hold. Assume also that $\phi \in L^\infty(\Omega_\ep)$ is a solution of \eqref{projected_problem} for a function $h \in C^{\alpha}(\bar \Omega_\ep)$. Then there is $C$ such that for $\ep>0$ sufficiently small
\begin{equation} \label{phi_estimate}
  \|\phi\|_{L^\infty(\Omega_\ep)} \le C \|h\|_\alpha
\end{equation}
and 
\begin{equation} \label{c_estimate}
  |c_{ij}| \le C \|h\|_\alpha.
\end{equation}
\end{lemma}
From now on, we denote by $C$ a generic constant which is independent of $\ep$ and the particulars $\xi_i'$, $\Lambda_i$ satisfying \eqref{constrains}. The proof of this lemma is based on the following non-degeneracy property of the solutions $w_{\lambda,\xi'}$ (see \cite{DadPSi2013}).
\begin{proposition}\label{non-degeneracy}
Any bounded solution $\phi$ of equation 
\begin{equation*}
(-\Delta)^s \phi  =  p^* w_{\lambda,\xi'}^{p^*-1} \phi
\quad \text{in } \R^n
\end{equation*}
is a linear combinations of the functions 
\begin{align}
\label{kernel}
\frac{n-2s}{2} w_{\lambda,\xi'}+ (x-\xi') \cdot \nabla w_{\lambda,\xi'}, \quad \frac{\partial w_{\lambda,\xi'}}{\partial \xi_j'} , \quad 1 \le j \le n .
\end{align}
\end{proposition}

We will also need the following elementary convolution estimate.
\begin{lemma}
\label{lconv}
For $2s<\alpha<n$ there is $C$ such that
\[
\|(1+|x|)^{\alpha-2s} ( \varGamma* h )\|_{L^\infty(\R^n)} \leq C \|(1+|x|)^\alpha h\|_{L^\infty(\R^n)} ,
\]
where $\varGamma$ is defined in \eqref{Gamma}.
\end{lemma}

\begin{proof}[Proof of Lemma~\ref{a_priori_bound}]
Let  $\xi_i' = \ep^{-\frac{1}{n-2s}}\xi_i \in \Omega_\ep$, $\lambda_i > 0$, $i=1,\dotsc,m$, denote the properly scaled points and the parameters, respectively. Let us first estimate the constants $c_{ij}$. Testing the first equation in \eqref{projected_problem} against $z_{lk}$ and then integrating by parts twice, we deduce that
\[
\sum_{i,j} c_{ij} \int_{\Omega_{\ep}} w^{p^*-1} z_{ij} z_{lk} = \int_{\Omega_{\ep}}[(-\Delta_{\Omega_\ep})^s z_{lk} - (p^* \pm \ep) \bar v^{p^*-1\pm \ep} z_{lk}] \phi -\int_{\Omega_{\ep}} h z_{lk}, \quad \ep > 0.
\]
This defines a linear system in the $c_{ij}$'s which is almost diagonal as $\ep$ approaches to zero, indeed, for $k=1,\dotsc,n$,
\[
  \int_{\Omega_{\ep}} w^{p^*-1} z_{ij} z_{lk} = \delta_{il} \delta_{jk} \int_{\R^n} w_{\lambda_i,0}^{p^*-1} \left( \frac{\partial w_{\lambda_i,0}}{\partial x_k}\right)^2 + o(1) 
\]
and for $k=0$ 
\[
 \int_{\Omega_{\ep}} w^{p^*-1} z_{ij} z_{l0} = \delta_{il} \delta_{j0} \int_{\R^n} w_{\lambda_i,0}^{p^*-1} \left( \frac{n-2s}{2} w_{\lambda_i,0} + x \cdot \nabla w_{\lambda_i,0} \right)^2 + o(1).
\]
On the other hand, we deduce that, for $l = 1,\dotsc,m$,
\[
\int_{\Omega_{\ep}}[(-\Delta_{\Omega_\ep})^s z_{lk} - (p^* \pm \ep) \bar v^{p^*-1\pm \ep} z_{lk}] \phi = o(1) \|\phi\|_{L^\infty(\Omega_{\ep})},
\]
after noticing that $(-\Delta)^s \bar z_{lk} - p^* w_{\lambda_l,0}^{p^*-1} \bar z_{lk} = 0$ (recall the definition of $\bar z_{lk}$ in \eqref{bar-z_a} and \eqref{bar-z_b}), and then applying the dominated convergence theorem. It is also easy to see that
\[
\left| \int_{\Omega_{\ep}} h z_{lk} \right| \le C \|h\|_\alpha.
\]
Therefore, the constants $c_{ij}$ satisfy the estimate
\begin{equation} \label{4.2}
|c_{ij}| \le C \|h\|_\alpha + o(1) \|\phi\|_{L^\infty(\Omega_{\ep})} \quad \text{as } \ep \to 0.
\end{equation}

We proceed by contradiction to prove \eqref{phi_estimate}. By abuse of notation, assume that there are a sequence $\ep \to 0$ and functions $\phi_\ep \in L^\infty(\Omega_{\ep})$, which are solution of \eqref{projected_problem} for some $h_\ep$, and such that 
\[
\|\phi_\ep\|_{L^\infty(\Omega_{\ep})}  = 1, \quad \| h_\ep\|_\alpha \to 0 \quad \text{as } \ep \to 0.
\]
Let us denote by $\xi_{i\ep}' = \ep^{-\frac{1}{n-2s}}\xi_i \in \Omega_\ep$ and $\lambda_{i\ep}$ the corresponding points and scalars associated to the previous sequence. Observe that, by \eqref{4.2},
\begin{equation} \label{c_estimate_2}
  |c_{ij}| \le C \|h_\ep\|_\alpha + o(1) \|\phi_\ep\|_{L^\infty(\Omega_{\ep})} = o(1) \quad \text{as } \ep \to 0.
\end{equation}
We shall prove that 
\begin{equation} \label{lim norm phin}
\lim_{\ep \to 0} \|\phi_\ep\|_{\gamma} =0,
\end{equation}
for any $\gamma = \min\{\alpha,\beta\}-2s$, where  $\beta$ is any number in the interval $(2s,4 s)$. In particular $\|\phi_\ep\|_{L^\infty(\Omega_{\ep})} \to 0$ as $\ep \to 0$, which is a contradiction. 

To show \eqref{lim norm phin}, we first prove that for any $R>0$, 
\begin{equation} \label{conv comp}
\phi_\ep \to 0 \quad \text{uniformly on } B_R(\xi_{i\ep}').
\end{equation}
Suppose that this is not true and translate the system of coordinates so that $\xi_{i\ep}'=0$. Then there is some point $x_\ep \in B_R(0)$ such that 
\begin{align}
\label{phin xn}
|\phi_\ep(x_\ep)|\ge \frac12 .
\end{align}
By passing to a subsequence we can assume that $\phi_\ep$ converges uniformly on compact sets of $\R^n$ to a bounded solution $\phi$ of the problem
\[
(-\Delta)^s \phi = p^* w_{\lambda,0}^{p^*-1}\phi
\quad\text{in } \R^n
\]
for some $\lambda>0$ (recall that $\lambda_{i\ep}$ stay bounded and bounded away from zero by \eqref{constrains}). 
By Proposition~\ref{non-degeneracy},  $\phi$ is a linear combination of the $z_{ij}$'s. We can take the limit in the third equation of \eqref{projected_problem} and use the Lebesgue dominated convergence theorem to find that  $\phi$ satisfies
\[
  \int_{\R^n} w_{\lambda,0}^{p^*-1} z_{ij} \phi = 0 \quad \text{for all } 1\le i \le m, 0 \le j \le n, 
\]
and we deduce from this that $\phi\equiv0$.
But due to \eqref{phin xn} there must be a point $x$ such that $|\phi(x)|\geq \frac12$, which is a contradiction.

On the other hand, we claim that for any $2s \le \beta< 4 s$
\begin{equation} \label{4.3}
\lim_{\ep \to 0} \| (p^* \pm \ep) \bar v^{p^*-1 \pm \ep} \phi_\ep\|_\beta = 0.
\end{equation}
Indeed, observe that $0< (p^* \pm \ep) \bar v^{p^*-1 \pm \ep} \le C \sum_{i=1}^m (1+|x-\xi_{i\ep}'|)^{-4s}$, so
\[
  \| (p^* \pm \ep) \bar v^{p^*-1\pm \ep} \phi_\ep\|_\beta \le C \sup_{x\in\Omega_{\ep}}
  \left(
     \frac{\sum_{i=1}^m 
     (1+|x-\xi_{i,\ep}'|)^{-4s}} {\sum_{i=1}^m (1+|x-\xi_{i,\ep}'|)^{-\beta}} |\phi_\ep(x)|
  \right).
\]
Let $\bar\ep>0$ be given. Then there exists $R>0$ large so that 
\[
  \frac{\sum_{i=1}^m (1+|x-\xi_{i\ep}'|)^{-4s}}{\sum_{i=1}^m (1+|x-\xi_{i\ep}'|)^{-\beta}} 
  \leq \bar\ep \quad \forall x \in \Omega_\ep \setminus \cup_{i=1}^m B_R(\xi_{i\ep}') .
\]
By \eqref{conv comp}, there is $\ep_0$ such that for all $\ep < \ep_0$
\[
\sup_{B_R(\xi_{i\ep}')} |\phi_\ep| \leq \bar \ep.
\]
It follows that for $\ep < \ep_0$,
\[
\sup_{x\in\Omega_\ep}\frac{\sum_{i=1}^m (1+|x-\xi_{i\ep}'|)^{-4s}}
{\sum_{i=1}^m (1+|x-\xi_{i\ep}'|)^{-\beta}} |\phi_\ep(x)| \leq \bar \ep ,
\]
and this proves \eqref{4.3}.

Let us now consider 
\begin{gather*}
  f_{i\ep}(x) = \frac{(1+|x-\xi_{i\ep}'|)^{-\beta}}{\sum_{j=1}^m (1+|x-\xi_{j\ep}'|)^{-\beta}} 
  |(p^* \pm \ep) \bar v^{p^*-1 \pm \ep} \phi_\ep(x)|, \\
  h_{i\ep}(x) = \frac{(1+|x-\xi_{i\ep}'|)^{-\alpha}}{\sum_{j=1}^m (1+|x-\xi_{j\ep}'|)^{-\alpha}} |h_\ep(x)|, \\
  t_{i\ep}(x) = \frac{(1+|x-\xi_{i\ep}'|)^{-\alpha}}{\sum_{j=1}^m (1+|x-\xi_{j\ep}'|)^{-\alpha}} |\sum_{l,k} c_{lk} w^{p^*-1} z_{lk}|,
\end{gather*}
and observe that $ \sum_i f_{i\ep}=| (p^* \pm \ep) \bar v^{p^*-1 \pm \ep} \phi_\ep| $, $ \sum_i h_{i\ep}=|h_\ep|$ and $\sum_i t_{i\ep} = \sum_{l,k} c_{lk} w^{p^*-1} z_{lk}$. We extend the functions $f_{i\ep}$, $h_{i\ep}$ and $t_{i\ep}$ by zero outside $\Omega_{\ep}$.
Let $\psi_{i\ep}$ be the solution to 
\[
(-\Delta)^s \psi_{i\ep} = f_{i\ep} + h_{i\ep}+ t_{i\ep}
\quad\text{in } \R^n,
\]
with $\psi_{i\ep}(x) \to 0$ as $|x| \to \infty$, obtained by convolution with $\varGamma$.

Let $\psi_\ep = \sum_i \psi_{i\ep}$ and observe that $\psi_\ep$ satisfies 
\[
  (-\Delta)^s \psi_\ep = g_\ep \quad \text{in } \R^n
\]
where
\begin{equation*}
  g_\ep(x) =
    \left\{
      \begin{aligned}
        |(p^* \pm \ep) \bar v(x)^{p^*-1\pm \ep} \phi_\ep(x)| &+ |h_\ep(x)| + |\sum_{l,k} c_{lk} w(x)^{p^*-1} z_{lk}(x)| && \text{if } x \in \Omega_{\ep}, \\
        & 0 && \text{if } x \not\in \Omega_{\ep}.
      \end{aligned}
    \right.
\end{equation*}
Using the maximum principle for the extended problem in $\Omega_{\ep} \times (0,\infty)$, Lemma~\ref{maximum_principle}, we find
\begin{align}
\label{ineq phi psi}
|\phi_\ep| \leq \psi_\ep
\quad \text{in } \Omega_{\ep}.
\end{align}
Therefore we can get weighted $L^\infty$ estimates for $\phi_\ep$ by establishing these for $\psi_\ep$.

Note that centering at $\xi_{i\ep}'=0$,
\[
  \|(1+|x|)^\alpha h_{i\ep}\|_{L^\infty} \leq \|h_\ep\|_\alpha 
\]
and therefore, by the previous lemma,
\begin{align} \label{a1}
\| ( 1+|x|)^{\alpha-2s} \varGamma*h_{i\ep}\|_{L^\infty(\R^n)}
\leq \|h_\ep\|_\alpha.
\end{align}
Similarly,
\begin{align} \label{a2}
\| ( 1+|x|)^{\alpha-2s} \varGamma*t_{i\ep}\|_{L^\infty(\R^n)} \leq \sum_{l,k}|c_{lk}|\|w^{p^*-1} z_{lk}\|_\alpha.
\end{align}
Finally, if $2s <\beta < 4s$, using again the previous lemma we find that
\begin{align} \label{a3}
\| ( 1+|x|)^{\beta-2s} \varGamma* f_{i\ep}\|_{L^\infty(\R^n)} \leq \| ( 1+|x|)^{\beta} f_{i\ep}\|_{L^\infty(\R^n)}
\leq \|(p^* \pm \ep) \bar v^{p^*-1 \pm \ep}\phi_\ep\|_\beta .
\end{align}
Hence, by \eqref{a1}--\eqref{a3}
\[
\|\psi_\ep\|_\gamma \leq C (\|(p^* \pm \ep) \bar v^{p^*-1\pm \ep}\phi_\ep\|_\beta +  \|h_\ep\|_\alpha + \sum_{l,k}|c_{lk}|\|w^{p^*-1} z_{lk}\|_\alpha) ,
\]
where $\gamma = \min\{\beta-2s,\alpha-2s\}$.
Using \eqref{ineq phi psi} we get that
\[
\|\phi_\ep\|_\gamma
\leq C (\|(p^* \pm \ep) \bar v^{p^*-1\pm \ep}\phi_\ep\|_\beta +  \|h_\ep\|_\alpha + \sum_{l,k}|c_{lk}|\|w^{p^*-1} z_{lk}\|_\alpha).
\]
But $\|(p^* \pm \ep) \bar v^{p^*-1\pm \ep}\phi_n\|_\beta +  \|h_\ep\|_\alpha + \sum_{l,k}|c_{lk}|\|w^{p^*-1} z_{lk}\|_\alpha  \to 0$ 
as $\ep \to 0$ by \eqref{c_estimate_2} and \eqref{4.3}. This proves \eqref{lim norm phin}.

Finally, \eqref{c_estimate} is a consequence of \eqref{phi_estimate} and \eqref{4.2}.

\end{proof}

As a consequence of Lemma~\ref{a_priori_bound}, we deduce the following proposition.

\begin{proposition} \label{proposition_linear_eq}
Let $\alpha \in (2s,4s)$ and assume constrains \eqref{constrains} hold.
Then there are numbers $\ep_0 > 0$, $C>0$, such that for all $0 < \ep < \ep_0$ and all $h \in C^{\alpha}(\bar \Omega_\ep)$, problem \eqref{projected_problem} admits a unique solution $\phi = L_{\ep}(h)$, and 
\begin{equation}\label{phi_estimate2}
  \|L_{\ep}(h)\|_{\alpha - 2s} \le C \|h\|_\alpha \, ,  
\end{equation}
\begin{equation}
  |c_{ij}| \le C \|h\|_\alpha \, .
\end{equation}
\end{proposition}

\begin{proof}
Let us consider the space
\[
  \mathcal H = \{ \phi \in H(\Omega_\ep): \int_{\Omega_\ep}\phi w^{p^*-1} z_{ij} = 0 \quad \forall i,j\}
\]
endowed with the usual inner product. The weak formulation of problem \eqref{projected_problem} is the following: Find $\phi \in \mathcal H$ such that
\[
  <\phi,\psi> = \int_{\Omega_\ep} (p \pm \ep) \bar v^{p^*-1\pm \ep} \phi \psi + \int_{\Omega_\ep}(h + \sum_{i,j} c_{ij} w^{p^*-1} z_{ij}) \psi \quad \text{for all } \psi \in H. 
\]
With the aid of the Riesz's representation theorem, this equation takes the form
\begin{equation} \label{4.4}
  \phi = \mathcal{F}_\ep (\phi) + \tilde h
\end{equation}
where $\mathcal{F}_\ep$ and $\tilde h$ are operators defined in $L^2(\Omega_\ep)$ by
\begin{gather*}
\mathcal{F}_\ep = (-\Delta_{\Omega_\ep})^{-s} \circ l_1, \\
\tilde h = (-\Delta_{\Omega_\ep})^{-s} \circ l_2;
\end{gather*}
$l_1$ and $l_2$ are the functions defined in $L^2(\Omega_\ep)$ given by
\begin{gather*}
  l_1 (\psi) = \int_{\Omega_\ep} (p^* \pm \ep) \bar v^{p^*-1\pm \ep} \phi, \\
  l_2 (\psi) = \int_{\Omega_\ep}(h + \sum_{i,j} c_{ij} w^{p^*-1} z_{ij}) \psi.
\end{gather*}
$(-\Delta_{\Omega_\ep})^{-s}$ represents the inverse of the fractional Laplacian operator. 

Using the compactness of the embedding of $H(\Omega_\ep)$ into $L^2(\Omega_\ep)$, we deduce that $\mathcal{F}_\ep$ is compact (see for instance \cite[Ch.~VII]{Ad1975} and \cite{DiNPaVa2012,BSV}). Fredholm's alternative guarantees unique solvability of this problem for any $h$ provided that the homogeneous equation
\[
 \phi = \mathcal{F}_\ep (\phi)
\] 
has only the zero solution in $\mathcal H$. Lemma~\ref{a_priori_bound} guarantees that this is true provided that $\ep > 0$ is small enough.

Finally, estimate \eqref{phi_estimate2} is a consequence of \eqref{phi_estimate} and a simple argument by contradiction.
\end{proof}

It is important for later purposes to understand the differentiability of $L_\ep$ on the variables $\xi_i'$ and $\Lambda_i$. To this end, given $\alpha \in (2s,4s)$, we define the space
\[
  L^{\infty}_\alpha (\Omega_\ep) = \{h \in L^{\infty}(\Omega_\ep): \|h\|_\alpha < \infty \},
\]
and consider the map
\[
  (\xi',\Lambda,h) \mapsto S(\xi',\Lambda,h) \equiv L_\ep(h),
\]
as a map with values in $L^{\infty}_\alpha \cap H(\Omega_
\ep)$. 

The proof of the next results are similar to that found in \cite{dPFeMu2003} for the case $m=2$ (see also \cite{dPFeMu2002}). We omit the details.

\begin{proposition}
Under the conditions of the previous proposition, the map $S$ is of class $C^1$ and
\[
  \|\nabla_{\xi',\Lambda} S(\xi',\Lambda,h)\|_{\alpha - 2s} \le C \|h\|_\alpha.
\]
\end{proposition}

\begin{proposition}\label{proposition_nonlinear_eq}
Assume the conditions of Proposition~\ref{proposition_linear_eq} are satisfied. Then, there is a constant $C>0$  such that, for all $\ep > 0$ small enough, there exists a unique solution
\[
 \phi = \phi(\xi',\Lambda)= \tilde \phi + \psi
\]
to problem \eqref{4.1} with $\psi$ defined by $\psi = L_\ep (R_\ep)$ and for points $\xi', \Lambda$ satisfying \eqref{constrains}. Moreover, the map $(\xi',\Lambda) \mapsto \tilde \phi(\xi',\Lambda)$ is of class $C^1$ for the  $\|\cdot\|_{\alpha-2s}$-norm and
\begin{gather}
\|\tilde \phi\|_{\alpha-2s} \le C\ep^{\min\{p^*,2\}}, \\
\|\nabla_{\xi',\Lambda} \tilde \phi\|_{\alpha-2s} \le C\ep^{\min\{p^*,2\}}.
\end{gather}
\end{proposition}


\section{The reduced functional}

Let us consider points $(\xi',\Lambda)$ which satisfy constrains \eqref{constrains} for some $\delta>0$, and recall that $\xi'= \ep^{-\frac{1}{n-2s}} \xi$. Let $\phi(x) = \phi(\xi',\Lambda)(x)$ be the unique solution of \eqref{4.1} given by Proposition~\ref{proposition_nonlinear_eq}. Let $\Phi$ the $s$-harmonic extension of $\phi$ (recall \eqref{extension}) and consider the functional 
\[
 \J(\xi,\Lambda) = J_{\pm \ep} (\bar V + \Phi),
\]
where $J_{\pm \ep}$ is defined in \eqref{energy_functional}. The definition of $\Phi$ yields that
\[
   \J'(\bar V + \Phi)[\Theta] = 0 \quad \text{for all } \Phi \in H_{0,L}^s(\Omega_\ep)
\]
such that 
\[
  \int_{\Omega_\ep}\theta w^{p^*-1} z_{ij},
\]
where $\theta = \tr|_{\Omega_\ep \times \{0\}} \Theta$.

It is easy to check that
\[
\partial_{x_j}{v_i} = z_{ij} + o(1), \quad \partial_{\Lambda_j} v_i = z_{i0} + o(1),
\]
as $\ep \to 0$. The last part of Proposition~\ref{proposition_nonlinear_eq} gives the validity of the following result, see \cite{dPFeMu2002,dPFeMu2003} for details.
\begin{lemma}\label{reduction}
$v = \bar v + \phi$ is a solution of problem \eqref{main_equation} if and only if $(\xi,\Lambda)$ is a critical point of $\J$. 
\end{lemma}

Next step is then to give an asymptotic estimate for $\J(\xi,\Lambda)$. As we expected, this functional and $J_\ep (\bar V)$ coincide up to order $o(\ep)$. The steps to proof this result are basically contained in \cite[Sec.~4]{dPFeMu2002} and \cite[Sec.~6]{dPFeMu2003}, we omit the details.

\begin{proposition}\label{reduced_energy_proposition}
We have the expansion
\begin{equation} \label{reduced_energy}
\J(\xi,\Lambda) = m C_{n,s} + [\gamma_{n,s} + \omega_{n,s} \Psi(\xi,\Lambda)] \ep + o(\ep),
\end{equation}
where $o(\ep) \to 0$ as $\ep \to 0$ in the uniform $C^1$-sense with respect to $(\xi,\Lambda)$ satisfying \eqref{xi_separation} and  \eqref{lambda_separation}. The constants in \eqref{reduced_energy} are those in Lemma~\ref{energy_expansion} and 
\[
  \Psi(\xi, \Lambda)= \frac{1}{2} \left\{\sum_{i=1}^m H(\xi_i,\xi_i) \Lambda_i^2-2\sum_{i<j} G(\xi_i,\xi_j)\Lambda_i \Lambda_j \right\} \pm \log(\Lambda_1 \dotsm \Lambda_m).
\]
\end{proposition}
This result together Lemma~\ref{energy_expansion} and its remark imply
\begin{equation} \label{reduced_energy_gradient}
  \nabla \J(\xi,\Lambda) = \omega_{n,s} \nabla \Psi(\xi,\Lambda) \ep + o(\ep). 
\end{equation} 
Lemma~\ref{reduction} and this estimate show up the importance of the stable critical points of $\Psi$ to find solutions of \eqref{main_equation}, and thereby to \eqref{1.1}.


\section{Proof of the main results}
 
We will now show how the results of previous sections imply the validity of the theorems stated in Section~1 (in the spectral fractional Laplacian case). Let us first note that Theorem~\ref{bubbling_solutions1} is a direct consequence of Lemma~\ref{reduction} together \eqref{reduced_energy_gradient} and the stability of the set $\A$ of critical points of $\Psi$.


\subsection{The supercritical case, two-bubble solutions}

In this subsection we prove Theorem~\ref{two-bubble_solutions}. So we consider the supercritical case, $p = p^* + \ep$, in \eqref{main_equation}, and look for a two-bubble solution. We set up a min-max scheme to find a critical point of the function $\Psi$ that will be used to find a critical point for the reduced functional $\J$, according to  Proposition~\ref{reduced_energy_proposition} and posterior comments.

In this setting the function $\Psi$ takes the form 
\begin{equation} \label{psi_two-bubble}
\Psi(\xi, \Lambda)= \frac{1}{2} \left\{H(\xi_1,\xi_1) \Lambda_1^2 + H(\xi_2,\xi_2) \Lambda_2^2 - 2 G(\xi_1,\xi_2)\Lambda_1 \Lambda_2 \right\} + \log(\Lambda_1 \Lambda_2),
\end{equation}
where $\xi = (\xi_1,\xi_2) \in \Omega^2$ and $\Lambda = (\Lambda_1,\Lambda_2) \in (0,\infty)^2$ satisfy \eqref{xi_separation} and \eqref{lambda_separation}, respectively. This function is well defined in $(\Omega^2 \setminus \Delta) \times (0,\infty)^2$, where $\Delta$ is the diagonal $\Delta = \{ (\xi_1,\xi_2) \in \Omega^2 : \xi_1 = \xi_2 \}$. We avoid the singularities of $\Psi$ over $\Delta$ by truncating the Green function as follows. For $M>0$, define
\[
  G_M(\xi) = 
  \begin{cases}
    G(\xi) & \text{if } G(\xi) \le M \\
    M      & \text{if } G(\xi) > M,
  \end{cases}
\]
and consider $\Psi_{M,\rho}: \Omega_\rho^2 \times (0,\infty)^2 \rightarrow \R$ given by 
\begin{equation} \label{Psi_modified}
  \Psi_{M,\rho}(\xi,\Lambda) = \Psi(\xi,\Lambda) - 2G_M(\xi)\Lambda_1\Lambda_2 + 2G_M(\xi)\Lambda_1\Lambda_2,
\end{equation}
where $\rho > 0$ and $\Omega_\rho = \{ \xi \in \Omega : \dist (\xi,\partial\Omega) > \rho \}$.
The quantities $M$ and $\rho$ will be chosen later, and we still denote by $\Psi$ the modified function $\Psi_{M,\rho}$.

For every $\xi \in \M^2$ we choose $\Lambda(\xi) = (\Lambda_1(\xi),\Lambda_2(\xi))$ to be a vector defining a negative direction of the quadratic form associated with $\Psi$. Such a direction exists since, by hypothesis, the function defined in \eqref{varphi_function},
\[
  \varphi(\xi_1,\xi_2) = H^{1/2}(\xi_1,\xi_1)H^{1/2}(\xi_2,\xi_2)-G(\xi_1,\xi_2),
\]
is negative over $\M^2$. More precisely, fixed $\xi = (\xi_1,\xi_2) \in \Omega^2$ and considering $\psi$ in \eqref{psi_two-bubble} as a function of $\Lambda = (\Lambda_1,\Lambda_2)$, we have a unique critical point $\Lambda (\xi) = (\Lambda_1(\xi),\Lambda_2(\xi))$ given by
\begin{equation} \label{lambda(xi)}
  \Lambda_1^2 = -\frac{H(\xi_2,\xi_2)^{1/2}}{H(\xi_1,\xi_1)^{1/2}\varphi(\xi_1,\xi_2)}, \quad \Lambda_2^2 = -\frac{H(\xi_1,\xi_1)^{1/2}}{H(\xi_2,\xi_2)^{1/2}\varphi(\xi_1,\xi_2)}. 
\end{equation}
In particular, we have that 
\[
  H(\xi_1,\xi_1) \Lambda_1^2 + H(\xi_2,\xi_2) \Lambda_2^2 - 2 G(\xi_1,\xi_2)\Lambda_1 \Lambda_2 = -1
\]
and 
\[
  \Psi(\xi,\Lambda(\xi)) = -\frac{1}{2} + \log \frac{1}{|\varphi(\xi)|}.
\]

In order to define the min-max class, we consider the set $D = \{ (\xi,\Lambda) \in \Omega_\rho^2 \times (0,\infty)^2 : \varphi(\xi) < -\rho_0 \}$, where $\rho_0 = \minimum \left\{ \frac{1}{2}\exp (-2C_0-1), -\frac{1}{2} \maximum \{ \varphi : \text{in } \M^2 \}  \right\}$, with 
\[
  C_0 = \sup_{(\xi,\sigma) \in \M^2 \times I} \Psi (\xi,\sigma);
\]
$I$ is the interval $(\sigma_0,\sigma_0^{-1})$ where $\sigma_0$ is a small number to be chosen later. With the previous choice of constants, we necessary verify that $\M^2 \times (0,\infty)^2 \subset D$. Now, let $\Z$ be the class of continuous functions $\zeta : \M^2 \times I \times [0,1] \rightarrow D$, such that
\begin{enumerate}[(i)]
\item $\zeta(\xi,\sigma_0,t) = (\xi,\sigma_0 \Lambda(\xi))$, and $\zeta(\xi,\sigma_0^{-1},t) = (\xi,\sigma_0^{-1} \Lambda(\xi))$ for all $\xi \in \M^2$, $t \in [0,1]$, and
\item $\zeta(\xi,\sigma,0) = (\xi,\sigma \Lambda(\xi))$ for all $(\xi,\sigma) \in \M^2 \times I$.
\end{enumerate}  
Then we define the min-max value
\begin{equation}\label{critical_value}
  c(\Omega) = \inf_{\zeta \in \mathcal{Z}} \sup_{(\xi,\sigma) \in \M^2 \times I} \Psi(\zeta(\xi,\sigma,1))
\end{equation}
and we will prove that $c(\Omega)$ is a critical value of $\Psi$.

To prove that $c(\Omega)$ is well-defined, it is necessary an intersection lemma that depends on a topological continuation of the set of solutions of an equation. The idea behind is based on the work of Fitzpatrick, Massab\`o and Pejsachowicz \cite{FiMaPe1983} (see Corollary~7.1 in \cite{dPFeMu2003}). For the ``fractional $\Psi$'', the proof of the next related result  is similar and we omit the details. We point out that the hypothesis in Theorem~\ref{two-bubble_solutions} about the topological nontriviality of the set $\varphi < 0$ is used precisely here. 

\begin{lemma} \label{bound_below}
There is a positive constant $K$, independent of $\sigma_0$, such that
\begin{equation}\label{k_constant}
  \sup_{(\xi,\sigma) \in \M^2 \times I} \Psi(\zeta(\xi,\sigma,1)) \ge -K \quad \text{for all } \zeta \in \Z.
\end{equation}
\end{lemma}   

The next step is to show that the domain in which $ \Psi $ is defined is closed for the gradient flow of this function. The following lemma support this claim.
\begin{lemma} \label{flow_closed}
Given $c<0$, there exists a sufficiently small number $\rho>0$ satisfying the following: If $(\xi_1, \xi_2) \in \partial(\Omega_\rho \times \Omega_\rho)$ is such that $\varphi(\xi_1, \xi_2) = c$, then there is a vector $\tau$, tangent to $\partial(\Omega_\rho \times \Omega_\rho)$ at the point $(\xi_1, \xi_2)$, so that
\begin{equation} \label{closed_gradient}
  \nabla \varphi (\xi_1, \xi_2) \cdot \tau \neq 0.
\end{equation}
The  number $\rho$ does not depend on $c$.
\end{lemma}
\begin{proof}
Fix $c<0$ and, for $\rho > 0$ small, suppose that $\xi_{1\rho} \in \partial \Omega_\rho$, $\xi_{2\rho} \in \Omega_\rho$ and $\varphi(\xi_{1\rho}, \xi_{2\rho}) = c$. After a rotation and a translation, we can assume that $\xi_{1\rho} = (0_{\R^{n-1}},\rho)$ and that the closest point of $\partial \Omega$ to $\xi_{1\rho}$ is the origin. To analyze the behavior of $\nabla \varphi (\xi_{1\rho}, \xi_{2\rho})$ as $\rho \to 0$, is convenient to consider the enlarged domain
\[
  \Omega^\rho = \rho^{-1}\Omega,
\]
and use the notation $\bar \xi = \rho^{-1}\xi \in \Omega^\rho$ for $\xi \in \Omega$. Observe that the associated Green function of $\Omega^\rho$ and its regular part are given by
\[
  G_\rho(\bar \xi_1,\bar \xi_2) = \rho^{n-2} G(\xi_1,\xi_2), \quad H_\rho(\bar \xi_1,\bar \xi_2) = \rho^{n-2} H(\xi_1,\xi_2),
\]
and then 
\[
  \varphi_\rho (\bar \xi_1,\bar \xi_2) = \rho^{n-2} \varphi (\xi_1,\xi_2),
\]
where
\[
\varphi_\rho (\bar \xi_1,\bar \xi_2) = H_\rho^{1/2}(\bar \xi_1,\bar \xi_1) H_\rho^{1/2}(\bar \xi_2,\bar \xi_2) - G_\rho (\bar \xi_1,\bar \xi_2).
\]

We denote by $\bar \xi_1$ the point $\bar \xi_{1\rho} = (0_{\R^{n-1}},1)$. We claim that $|\bar \xi_1 - \bar \xi_{2\rho}| = O(1)$ as $\rho \to 0$. Indeed, by contradiction, suppose the this is not true. Arguing as in the proof of Lemma~\ref{robin_function_lemma}, we deduce that there exists a constant $C>0$ such that
\begin{gather*}
 H_\rho^{1/2}(\bar \xi_1,\bar \xi_1) H_\rho^{1/2}(\bar \xi_{2\rho},\bar \xi_{2\rho}) \ge C |\bar \xi_1 - \bar \xi_{2\rho}|^{-(n-2)/2}, \\
 G_\rho (\bar \xi_1,\bar \xi_{2\rho}) \le C |\bar \xi_1 - \bar \xi_{2\rho}|^{-(n-2)}.
\end{gather*}
 Therefore, for $\rho>0$ sufficiently small, 
 \[
   C |\bar \xi_1 - \bar \xi_{2\rho}|^{-(n-2)/2} \le \varphi_\rho (\bar \xi_1,\bar \xi_{2\rho}) = c \rho^{n-2},
 \] 
which is a contradiction since $c<0$. We note that this conclusion only depends on the fact that $c$ is negative.

Then we can assume that $\bar \xi_{2\rho} \to \bar \xi_2$ as $\rho \to 0$, for certain $\bar \xi_2 = (\bar \xi_2',\bar \xi_2^{n})$, where $\xi_2^{n} \ge 1$. Observe that as $\rho \to 0$ the domain $\Omega^\rho$ becomes the half-space $\R^n_+ = \{ \xi = (\xi^ 1,\dotsc,\xi^n) \in \R^n : \xi^n > 0 \}$. Arguing again as in the proof of Lemma~\ref{robin_function_lemma}, the functions $H_\rho$ and $G_\rho$ converge to the corresponding ones $H_+$ and $G_+$ in the half-space $\R^n_+$, namely
\begin{gather*}
H_+(\xi_1,\xi_2) = \frac{a_{n,s}}{|\xi_1 - \hat \xi_2|^{n-2s}} \quad \text{and} \\
G_+(\xi_1,\xi_2) = a_{n,s} \left( \frac{1}{|\xi_1 - \xi_2|^{n-2s}} - \frac{1}{|\xi_1 - \hat \xi_2|^{n-2s}} \right);
\end{gather*}
where $\hat \xi_2 = (\xi_2',-\xi_2^n)$, for $\xi_2 = (\xi_2',\xi_2^n)$. Similarly $\varphi_\rho$ and its gradient converge to $\varphi_+$ and its gradient, respectively, where 
\[
\varphi_+ (\xi_1,\xi_2) = H_+^{1/2}(\xi_1,\xi_1) H_+^{1/2}(\xi_2,\xi_2) - G_+ (\xi_1,\xi_2).
\]

Now, since $\varphi_\rho (\bar \xi_1,\bar \xi_{2\rho}) = c \rho^{n-2}$, we have
\begin{equation} \label{varphi_xibar}
\varphi_+(\bar \xi_1,\bar \xi_2) = 0.
\end{equation}
Assume first that $\bar \xi_2' \neq 0$. Then, for the direction $ \tau = (0_{\R^n},\xi_2',0) $,
\[
  \nabla \varphi_+(\bar \xi_1,\bar \xi_2) \cdot \tau = -(n-2s)a_{n,s} \left( \frac{1}{|\bar \xi_1 - \bar \xi_2|^{n-2s+2}} - \frac{1}{|\bar \xi_1 - \hat{\bar \xi}_2|^{n-2s+2}} \right) |\bar \xi_2'|^2 \neq 0,
\]
since $|\bar \xi_1 - \bar \xi_2| < |\bar \xi_1 - \hat{\bar \xi}_2|$. Observe that for $\rho$ sufficiently small, $\tau = (0_{\R^n},\xi_2',0)$ is tangent to $\partial (\Omega_\rho \times \Omega_\rho)$ in $(\xi_{1\rho}, \xi_{2\rho})$. Assume now that $\bar \xi_2' = 0$, and suppose that $\bar \xi_2 = (0,\theta_0)$ with $\theta_0 > 1$. Consider the function $\varphi_+(\theta) = \varphi_+(\bar \xi_1,0_{\R^{n-1}},\theta) = \varphi_+(0_{\R^{n-1}},1,0_{\R^{n-1}},\theta)$; let us prove that $\varphi'_+(\theta_0) \neq 0$. Indeed, 
\[
a_{n,s}^{-1} \varphi_+(\theta) = \frac{1}{2^{n-2s} \theta^{(n-2s)/2}} - \frac{1}{(\theta-1)^{n-2s}} + \frac{1}{(\theta+1)^{n-2s}},
\]
and thereby 
\begin{equation}\label{varphi_prime}
a_{n,s}^{-1} \varphi_+'(\theta_0) = (n-2s) \left[  \frac{1}{(\theta_0-1)^{n+1-2s}} - \frac{1}{(\theta_0+1)^{n+1-2s}} + \frac{1}{2^{n+1-2s} \theta_0^{(n+2-2s)/2}} \right]
\end{equation}
On the other hand, \eqref{varphi_xibar} implies $\varphi_+(\theta_0) = 0$. Thus
\[
\frac{1}{2^{n-2s} \theta_0^{(n-2s)/2}} = \frac{1}{(\theta_0-1)^{n-2s}} - \frac{1}{(\theta_0+1)^{n-2s}},
\]
and putting this in \eqref{varphi_prime}, we deduce that $\varphi'_+(\theta_0) > 0$, as claimed. And then we find that 
\[
\nabla \varphi_+(\bar \xi_1,\bar \xi_2) \cdot \tau > 0,
\]
where $\tau = (0_{\R^n},0_{\R^{n-1}}, 1)$. Observe that $\tau$ is tangent to $\partial (\Omega_\rho \times \Omega_\rho)$ in $(\xi_{1\rho}, \xi_{2\rho})$, and the proof is complete.

\end{proof}

Finally, we have a critical value for $\Psi$.

\begin{proposition}
The number $c(\Omega)$ in \eqref{critical_value} is a critical value for $\Psi$ in $D$.
\end{proposition}

\begin{proof}
Let us first prove that for every sequence $\{\xi_j,\Lambda_j\}_j \subset D$ such that $(\xi_j,\Lambda_j) \to (\xi_0,\Lambda_0) \in \partial D$ and $\psi(\xi_j,\Lambda_j) \to c(\Omega)$ there is a vector $T$, tangent to $\partial D$ at $(\xi_0,\Lambda_0)$, such that
\begin{equation} \label{tangent}
\nabla \psi (\xi_0,\Lambda_0) \cdot T \neq 0.
\end{equation}
Now, since the function $\Psi (\xi,\Lambda)$ tends to $-\infty$ as $\Lambda$ is close to $\partial (0,\infty)^2$, we can assume that $\Lambda_0 \in (0,\infty)^2$, $\xi_0 \in \bar \Omega_\rho \times \bar \Omega_\rho$ and $\varphi(\Lambda_0) \le -\rho_0$. If $\nabla_\Lambda \psi(\xi_0,\Lambda_0) \neq 0$, choose $T = (0_{\R^{2n}},\nabla_\Lambda \psi(\xi_0,\Lambda_0))$. Otherwise, if $\nabla_\Lambda \psi(\xi_0,\Lambda_0) = 0$ then $\Lambda_0 = \Lambda(\xi_0)$ according to \eqref{lambda(xi)}, and
\begin{equation} \label{5}
\psi(\xi_0,\Lambda_0) = -\frac{1}{2} + \log \frac{1}{|\varphi(\xi_0)|}.
\end{equation}
Thus, $\varphi(\xi_0) = - \exp (-2c(\Omega)-1) \le -2\rho_0 <\rho_0$, so that $\xi_0 \in \partial (\Omega_\rho \times \Omega_\rho)$. Choosing $\rho > 0$ as in the previous lemma and then applying \eqref{closed_gradient}, we deduce \eqref{tangent} for certain direction $T$. To conclude, we choose $M > 0$ big enough: Let $M_\rho = \maximum \{ H(\xi,\xi) : \xi \in \Omega_\rho \}$, and consider $M \ge \exp (2K-1) + M_\rho$, where $K$ is the number found in Lemma~\ref{bound_below}. Using \eqref{k_constant} and \eqref{5}, we deduce that $G(\xi_0) \le M$ and thus $G_M = G$ near to $\xi_0$.

We can now define an appropriate gradient flow that will remain in $D$ at level $c(\Omega)$. Finally, let us check the Palais-Smale condition in $D$ at level $c(\Omega)$. Indeed, given the sequence $\{\xi_j,\Lambda_j\}_j \subset D$ satisfying $\psi(\xi_j,\Lambda_j) \to c(\Omega)$ and $\nabla \psi (\xi_j,\Lambda_j) \to 0$, we have that $\{\xi_j,\Lambda_j\}_j$ has a convergent subsequence since $\{ \Lambda_j \}_j$ is in fact bounded.
\end{proof}

\begin{proof}[Proof of Theorem~\ref{two-bubble_solutions}: the spectral fractional Laplacian case]
To complete the proof of the theorem, let us show how to find a critical value for $\J$ from the one for $\Psi$, namely $c(\Omega)$. We consider the domain $D_{r,R} = \Omega_\rho \times \Omega_\rho \times [r,R]^2 \cap D$. As we did with $\Psi$ at the beginning of this subsection, the functional $\J$ can be extended to all $D_{r,R}$ keeping the relations \eqref{reduced_energy} and \eqref{reduced_energy_gradient} over $D_{r,R}$.

By the Palais-Smale condition for $\Psi$ proved in the previous proposition, there are numbers $R>0$, $c>0$ and $\alpha_0 > 0$ such that for all $0<\alpha<\alpha_0$, and $(\xi,\Lambda) \in D_{r,R}$ satisfying $\Lambda > R$ and $c(\Omega)-2\alpha \le \Psi(\xi,\Lambda) \le c(\Omega)+2\alpha$ we have $|\nabla \Psi(\xi,\Lambda)| \ge c$. On the other hand, the min-max characterization of $c(\Omega)$ provides the existence of a $\zeta \in \Z$ such that
\[
  c(\Omega) \le \sup_{(\xi,\sigma) \in \M^2 \times I} \Psi(\zeta(\xi,\sigma,1)) \le c(\Omega)+\alpha.
\]
Choosing $r$ small and $R$ large if necessary, we can assume that $\zeta(\xi,\sigma,1) \in D_{r/2,R/2} \subset D_{r,R}$ for all $(\xi,\sigma) \in \M^2 \times I$.

We define a min-max value for $\J$ in the following way: Consider $\eta : D_{r,R} \times [0,\infty] \to D_{r,R}$ being the solution of the equation $\dot \eta = - h(\eta) \nabla \J (\eta)$ with initial condition $\eta(\xi,\Lambda,0) = (\xi,\Lambda)$. Here the function $h$ is defined in $D_{r.R}$ so that $h(\xi,\Lambda)=0$ for all $(\xi,\Lambda)$ with $\Psi(\xi,\Lambda) \le c(\Omega)-2\alpha$ and $h(\xi,\Lambda) = 1$ if $\Psi(\xi,\Lambda) \ge c(\Omega)-\alpha$, satisfying $0 \le h \le 1$. Since the choice of $r$ and $R$ and \eqref{reduced_energy}, \eqref{reduced_energy_gradient}, we have $\eta(\xi,\Lambda,t) \in D_{r,R}$ for all $t \ge 0$. Then the number 
\[
  C(\Omega) = \inf_{t \ge 0} \sup_{(\xi,\sigma) \in \M^2 \times I} \J(\eta(\zeta(\xi,\sigma,1),t))
\]
is a critical value for $\J$, and the proof is complete. 
\end{proof}


\subsection{The subcritical case, one-bubble solutions}

In this subsection we prove Theorem~\ref{theorem2}. Let us then suppose that $p=p^*-\ep$ in \eqref{main_equation} and $m=1$, that is, we consider the subcritical case and study the concentration phenomena for just one bubble. In this case the function $\Psi$ in \eqref{psi} takes the form
\[
  \Psi(\xi, \Lambda)= \frac{1}{2} H(\xi,\xi) \Lambda^2 - \log \Lambda, \quad \xi \in \Omega, \Lambda>0.
\]
Thanks to the coercivity of $\Psi$ in $\Lambda$, in order to find a critical point of $\Psi (\xi,\Lambda)$, we have to find one to $R(\xi) = H(\xi,\xi)$, that is, the Robin's function of the domain $\Omega$.

The next result shows that the Robin's function blows up at the boundary, which implies that  its absolute minimums are somehow stable under small variations of it. Before the precise statement of the result, let us review a fractional version of the Kelvin transform (see Appendix~A in \cite{RoSe2013}).

\begin{lemma}[Fractional Kelvin transform]
Let $u$ be a smooth bounded function in $\R^n \setminus \{0\}$. Let $\xi \mapsto \xi/|\xi|^2$ be the inversion with respect to the unit sphere. Define $u^*(\xi)=|\xi|^{2s-n} u(\xi^*)$. Then,
\[
  (-\Delta)^s u^*(\xi) = |\xi|^{-2n-s}(-\Delta)^s u(\xi^*),
\]
for all $\xi \neq 0$.
\end{lemma}
Recall also the following identity
\begin{equation}\label{inversions}
|\xi_1^*-\xi_2^*|=\frac{|\xi_1-\xi_2|}{|\xi_1||\xi_2|}.
\end{equation}

\begin{lemma} \label{robin_function_lemma}
Given $\xi \in \Omega$, we define the function $d(\xi) := \dist(\xi, \partial \Omega)$. Then, there exists positive constants $c_1$ and $c_2$ such that,
\begin{equation} \label{blow-up}
  c_1 d(\xi)^{2s-n} \le R(\xi) \le c_2 d(\xi)^{2s-n} \quad \text{for all } \xi \in \Omega.
\end{equation}
\end{lemma}

\begin{proof}
Let $\xi_0=(\xi_0^1,\dotsc,\xi_0^n) \in \partial \Omega$, and consider the ball $B := B_{1/2}(1/2,0,\dotsc,0) \subset \R^n$. After a rearrange of variables, we can assume that $ \xi_0 = (1,\dotsc,0) $ and $B \subset \Omega^c$. We shall use the Green function of $S_- = \{ (\xi^1,\dotsc,\xi^n) \in \R^n : \xi^1 < 1 \}$ and the Kelvin transform to bound from above the Green function of $\Omega$, which we denote by $G$.

Notice that the fractional Green function of the half-space $S_-$ (recall \eqref{greens_function}) is given by
\[
G_{S_-} (Z,Y) = \varGamma(Z-Y)-\varGamma(Z-\bar Y), \quad Z, Y \in \bar S_- \times [0,\infty), Z \neq Y,
\]
where $\bar Y$ is the reflection of $Y$ with respect to the half-plane $\partial S_- \times [0,\infty)$. Observe that $\Omega \times (0,\infty) \subset B^c \times (0,\infty) \ \subset S_-^*$. Then, we consider the $(n+1)$-dimensional Kelvin transform of the Green function of $S_-$ and define
\[
  F(Z,\xi) = |\xi|^{2s-n} |Z|^{2s-n}\left[ \varGamma(Z^*-\xi^*)-\varGamma(Z^*-\bar \xi^*)\right], \quad Z\in B^c \times (0,\infty), \xi \in B^c.
\]
It is easy to check that $F(Z,\xi)\ge 0$ on $\partial \Omega \times (0,\infty)$, and after using \eqref{inversions}, $F$ can be written as
\[
  F(Z,\xi) =  \varGamma(Z-\xi) - a_{n,s} \left| Z|\bar \xi^*|-\frac{\bar \xi^*}{|\bar \xi^*|} \right|^{2s-n}. 
\]

$F$ satisfies up to a positive constant
\begin{equation}
  \left\{
    \begin{aligned}
       \divr( y^{1-2s} \nabla  F(\cdot,\xi)) &= 0 && \text{in } \Omega \times (0,\infty), \\
       F(\cdot,\xi) &\ge 0 && \text{on } \partial \Omega \times (0,\infty), \\
       -\lim_{y\to0} y^{1-2s} \partial_y F(\cdot,\xi) &= \delta_\xi(\cdot) && \text{on } \Omega.
    \end{aligned}
  \right.
\end{equation}
Then, by a minor variant of the maximum principle (Lemma~\ref{maximum_principle}), we deduce that 
\[
  G(Z,\xi) \le F(Z,\xi) \quad \text{for all } Z \in \Omega \times (0,\infty),\ \xi \in \Omega.
\]
This implies that
\[
H(Z,\xi) \ge a_{n,s} \left| Z|\bar \xi^*|-\frac{\bar \xi^*}{|\bar \xi^*|} \right|^{2s-n} \quad \text{for all } Z \in \Omega \times (0,\infty),\ \xi \in \Omega.
\]
Thus, there exists a positive constant $c_1>0$ such that for all $\xi \in \Omega$ close to $\xi_0$,
\[
  R(\xi) = H(\xi,\xi) \ge c_1|\xi - \xi_0|^{2s-n}.
\]
Therefore, taking into account that $\xi_0 \in \partial \Omega$ is arbitrary, we conclude that in a neighborhood of $\partial \Omega$ there exist a constant $c_1 > 0$ such that $R(\xi) \ge c_1 d(\xi)^{2s-n}$. The smoothness of $H$ in $ \Omega $ allows us to extend this inequality to the whole domain. 

The other inequality in \eqref{blow-up} can be proven by a similar argument using an interior ball instead. 
\end{proof}

\begin{proof}[Proof of Theorem~\ref{theorem2}: the spectral fractional Laplacian case]
Thanks to the previous lemma, there still exist absolute minimums of small perturbations of $ R(\xi) = H(\xi,\xi)$. Theorem~\ref{theorem2} is then a consequence of this fact together Lemma~\ref{reduction} and \eqref{reduced_energy_gradient}.
\end{proof}


\section{The case of the restricted fractional Laplacian}

This section is devoted to the restricted fractional Laplacian and the necessary changes compared to the case of the spectral fractional Laplacian to handle it. The main changes will be in the stability of the critical points of $\Psi$. Most of the computations are however very similar and we leave some details to the reader.

Let us start recalling the definition of the restricted fractional Laplacian: 
\[
  (-\Delta_{|\Omega})^s u = c_{n,s}\mbox{
  P.V.}\int_{\mathbb{R}^n} \frac{\bar u(x) - \bar u(z)}{|x-z|^{n+2s}}\,\di z,
\]
With this operator, problem \eqref{1.1} reads as follows

\begin{equation} \label{fractional_problem_restricted}
 \left\{
  \begin{aligned}
    (-\Delta_{|\Omega})^s u & = u^{p^* \pm \ep} , \ 
   u>0\ \text{in } \Omega,\\
   & u = 0 \quad \text{in } \R^n \setminus \Omega,
  \end{aligned}
 \right.
\end{equation}
where $p^* = (n+2s)/(n-2s)$.

We recall Section~1 and denote by $G$ the Green function related to the restricted fractional Laplacian, namely the unique solution to
\begin{equation} \label{greenfun_restricted}
  \left\lbrace 
    \begin{aligned}
      (-\Delta_{|\Omega})^s G(\cdot,\xi) &= \delta_\xi (\cdot) && \text{in } \Omega, \\
          G(\cdot,\xi) &= 0                  && \text{in } \R^n \setminus \Omega.
    \end{aligned}
  \right. 
\end{equation}
The regular part of the Green function is defined by
\[
  H(x,\xi) = \varGamma(x,\xi) - G(x,\xi) \quad \text{for } x,\xi \in \Omega,\ x \not= \xi,
\]
where 
\[
  \varGamma(x,\xi) = \frac{a_{n,s}}{|x-\xi|^{n-2s}}
\]
is the Green function in the entire space $ \R^n $. 

As already noticed, the restricted fractional Laplacian is a self-adjoint operator on $L^2(\Omega)$ with discrete spectrum $\lambda_{k,s}$ and eigenfunctions $\phi_{k,s}$. Denote, as before, the Hilbert space 
\[
  H(\Omega)=\{u=\sum_{k=1}^\infty u_k \phi_{s,k} \in L^2(\Omega)\; : \; \| u \|^2_{H} = \sum_{k=1}^\infty \lambda_{s,k}\vert u_k\vert^2 <+\infty\}\subset L^2(\Omega)
\]

As for the spectral fractional Laplacian, a crucial tool is the Caffarelli-Silvestre extension, which in this case is simpler to state since it holds in all of $\R^{n+1}_+$ and not on the cylinder $\mathcal C$. In this case, problem \eqref{fractional_problem_restricted} writes 

\begin{equation*}
  \left\{
    \begin{aligned}
       \divr( y^{1-2s} & \nabla U ) = 0 && \text{in } \R^{n+1}_+ \\
       U &= u && \text{on } \Omega,\\
       U &= 0 && \text{on } \R^n \backslash \Omega \subset \R^{n+1}_+,\\
    \end{aligned}
  \right.
\end{equation*}
for $U\in \mathcal H^s(\R^{n+1}_+)$, the completion of $C^\infty_0(\overline{\R^{n+1}_+})$ with respect to the semi-norm 
$$
\Big (\int_{\R^{n+1}_+}y^{1-2s}|\nabla U|^2\Big )^{1/2},
$$
where $U$ vanishes outside of $\Omega \times (0,\infty)$. Then, up to a multiplicative constant,
\[
  (-\Delta_{| \Omega})^s u  = - \lim_{y \to 0} y^{1-2s} \partial_y U. 
\]

\subsection{Main preliminary results}

Let $\Omega$ be a bounded domain with smooth boundary in $\R^n$. As in Section~3, it is convenient to work with the enlarged domain 
\[
  \Omega_\ep = \ep ^{-\frac{1}{n-2s}}\Omega, \quad \ep > 0 \text{ small},
\]
that, after the change of variables 
\[
  v(x) = \ep ^{\frac{1}{2s \pm \ep (n-2s/2)}}u(\ep ^{\frac{1}{n-2s}}x), \quad x \in \Omega_\ep,
\]
transforms equation \eqref{fractional_problem_restricted} into
\begin{equation} \label{main_equation_restricted}
  \left\{
    \begin{aligned}
      (-\Delta_{| \Omega_\ep})^s v & = v^{p^* \pm \ep}, \  v >  0 \quad \text{in }\Omega_\ep, \\ 
      v & = 0 \quad \text{in } \R^n \setminus \Omega_\ep.
  \end{aligned}
\right.
\end{equation}

We develop again an initial approximation with concentration at certain $m$ points $\xi_1,\dotsc,\xi_m \in \Omega$, uniformly separated and away from the boundary of $\Omega$ as in \eqref{xi_separation}; this construction is based on the functions $w_{\lambda,\xi}$ in \eqref{w_definition}. To this end, we consider the properly scaled points 
\[
  \xi_i' = \ep^{-\frac{1}{n-2s}} \xi_i \in \Omega_\ep,
\]
and, for parameters $\lambda_1,\dotsc,\lambda_m>0$, look for a solution of problem \eqref{main_equation_restricted} of the form 
\[
  v = \bar v + \phi,
\]
where
\[
  \bar v = \sum_{i=1}^m v_i, \quad \text{with } v_i = v_{\lambda_i,\xi_i'}
\]
(the functions $v_{\lambda_i,\xi_i'}$ are the $H(\Omega_\ep)$-projection of $w_{\lambda_i,\xi_i'}$, as in \eqref{v_i}).

Thereby, our problem appears to be a critical point of the energy functional
\[
  J_{\pm \ep}(V) = \frac12 \int_{\R^{n+1}_+ } y^{1-2s}|\nabla V|^2-\frac{1}{p^*+1 \pm \ep} \int_{\Omega_\ep} |V|^{p^*+1 \pm \ep}.
\]
Since $V$ vanishes outside of $\Omega$, the same integration by parts arguments give the same result as in the expansion in Lemma \ref{energy_expansion} that we reproduce here for sake of completeness. 

\begin{lemma}
The following expansion holds:
\begin{equation}
J_{\pm \ep}(\bar V )= m C_{n,s} + [\gamma_{n,s} + \omega_{n,s} \Psi(\xi,\Lambda)]\ep + o(\ep)
\end{equation}
uniformly with respect to $(\xi,\Lambda)$ satisfying \eqref{xi_separation} and \eqref{lambda_separation}. Here
\begin{equation} \label{psi_restricted}
  \Psi(\xi, \Lambda)= \frac{1}{2} \left\{\sum_{i=1}^m H(\xi_i,\xi_i) \Lambda_i^2-2\sum_{i<j} G(\xi_i,\xi_j)\Lambda_i \Lambda_j \right\} \pm \log(\Lambda_1 \dotsm \Lambda_m),
\end{equation}
\[
  \gamma_{n,s} = \left\{ \pm \frac{m}{p^*+1}\omega_{n,s} \pm \frac{m}{2}\omega_{n,s} \log \beta_{n,s} \mp \frac{m}{p^*+1} \int_{\R^n} w^{p^*+1} \log w \right\} 
\]
and
\[
  \omega_{n,s} = \frac{\int_{\R^n} w^{p^*+1}}{p^*+1}.
\]
\end{lemma} 

The finite-dimensional reduction is completely similar to Section~4 and 5: for $\ep > 0$ small one can solve \eqref{main_equation_restricted} in suitable weighted spaces and find a result analogous to Proposition~\ref{proposition_nonlinear_eq}. In this way, we define the reduced function
\[
  \J (\xi,\Lambda) = J_{\pm \ep}(V),
\] 
where $V$ is the $s$-harmonic extension of the $m$-bubble solution of \eqref{main_equation_restricted} just found. The next results show that this reduced functional is the connection between the existence of solutions of \eqref{main_equation_restricted} and the existence of stable critical points of $\Psi$.
\begin{lemma}\label{reduction_restricted}
$v = \bar v + \phi$ is a solution of problem \eqref{main_equation_restricted} if and only if $(\xi,\Lambda)$ is a critical point of $\J$. 
\end{lemma}

\begin{proposition}
We have the expansion
\[
  \J(\xi,\Lambda) = m C_{n,s} + [\gamma_{n,s} + \omega_{n,s} \Psi(\xi,\Lambda)] \ep + o(\ep),
\]
where $o(\ep) \to 0$ as $\ep \to 0$ in the uniform $C^1$-sense with respect to $(\xi,\Lambda)$ satisfying \eqref{xi_separation} and  \eqref{lambda_separation}. Moreover
\begin{equation} \label{reduced_energy_gradient_restricted}
  \nabla \J(\xi,\Lambda) = \omega_{n,s} \nabla \Psi(\xi,\Lambda) \ep + o(\ep). 
\end{equation} 
\end{proposition}

This reduction scheme used to study the concentration phenomenon of solutions to \eqref{fractional_problem_restricted} makes clear the importance of finding stable critical sets (see Definition~\ref{stable_critical_set}) of $\Psi(\xi,\Lambda)$, where $\xi = (\xi_1,\dotsc,\xi_m) \in \Omega^m$ and $\Lambda = (\Lambda_1,\dotsc,\Lambda_m) \in (0,\infty)^m$. In fact, observe that Theorem~\ref{bubbling_solutions1} is a direct consequence of Lemma~\ref{reduction_restricted} together \eqref{reduced_energy_gradient_restricted}  and the stability of the set $\A$ of critical points of $\Psi$.


\subsection{The supercritical case, two-bubble solutions} 

In this subsection we prove Theorem~\ref{two-bubble_solutions} using a min-max argument, as in Section~6. So we consider the supercritical case $p^* + \ep$ in \eqref{fractional_problem_restricted} and look for a two-bubble solution.

In this setting the function $\Psi$ in \eqref{psi_restricted} takes the form 
\begin{equation} \label{psi_two-bubble_restricted}
\Psi(\xi, \Lambda)= \frac{1}{2} \left\{H(\xi_1,\xi_1) \Lambda_1^2 + H(\xi_2,\xi_2) \Lambda_2^2 - 2 G(\xi_1,\xi_2)\Lambda_1 \Lambda_2 \right\} + \log(\Lambda_1 \Lambda_2),
\end{equation}
where $\xi = (\xi_1,\xi_2) \in \Omega^2$ and $\Lambda = (\Lambda_1,\Lambda_2) \in (0,\infty)^2$ satisfy \eqref{xi_separation} and \eqref{lambda_separation}, respectively.
We develop a min-max scheme analog to that in the previous section, and define for $ \rho > 0 $ the set $ \Omega_\rho = \{ \xi \in \Omega : \dist (\xi,\partial\Omega) > \rho \} $. $\Psi$ should be properly modified, as in \eqref{Psi_modified}, to avoid its singularities. 

Recall the function defined in \eqref{varphi_function},
\[
  \varphi(\xi_1,\xi_2) = H^{1/2}(\xi_1,\xi_1)H^{1/2}(\xi_2,\xi_2)-G(\xi_1,\xi_2),
\]
and with it construct a min-max class of functions $\Z$ as in Subsection~6.1. Then define the value
\begin{equation}\label{critical_value_restricted}
  c(\Omega) = \inf_{\zeta \in \mathcal{Z}} \sup_{(\xi,\sigma) \in \M^2 \times I} \Psi(\zeta(\xi,\sigma,1)).
\end{equation}
This quantity turns out to be a critical value of $\Psi$.

The following lemma proves that $ c(\Omega) $ is well-defined. Its proof is similar to the one of Corollary~7.1 in \cite{dPFeMu2003}, we omit the details.

\begin{lemma}
There is a positive constant $K$, independent of $\sigma_0$, such that
\[
  \sup_{(\xi,\sigma) \in \M^2 \times I} \Psi(\zeta(\xi,\sigma,1)) \ge -K \quad \text{for all } \zeta \in \Z.
\]
\end{lemma}

On the other hand, the domain in which $ \Psi $ is defined is closed for the gradient flow of this function. This is a consequence of the following lemma, which is similar to Lemma~\ref{flow_closed}.

\begin{lemma}
Given $c<0$, there exists a sufficiently small number $\rho>0$ satisfying the following: If $(\xi_1, \xi_2) \in \partial(\Omega_\rho \times \Omega_\rho)$ is such that $\varphi(\xi_1, \xi_2) = c$, then there is a vector $\tau$, tangent to $\partial(\Omega_\rho \times \Omega_\rho)$ at the point $(\xi_1, \xi_2)$, so that
\[
  \nabla \varphi (\xi_1, \xi_2) \cdot \tau \neq 0.
\]
The  number $\rho$ does not depend on $c$.
\end{lemma}

\begin{proof}
Fix $c<0$ and, for $\rho > 0$ small, suppose that $\xi_{1\rho} \in \partial \Omega_\rho$, $\xi_{2\rho} \in \Omega_\rho$ and $\varphi(\xi_{1\rho}, \xi_{2\rho}) = c$. After a rotation and a translation, we can assume that $\xi_{1\rho} = (0_{\R^{n-1}},\rho)$ and that the closest point of $\partial \Omega$ to $\xi_{1\rho}$ is the origin. To analyze the behavior of $\nabla \varphi (\xi_{1\rho}, \xi_{2\rho})$ as $\rho \to 0$, is convenient to consider the enlarged domain
\[
\Omega^\rho = \rho^{-1}\Omega,
\]
and use the notation $\bar \xi = \rho^{-1}\xi \in \Omega^\rho$ for $\xi \in \Omega$. Observe that the associated Green function of $\Omega^\rho$ and its regular part are given by
\[
  G_\rho(\bar \xi_1,\bar \xi_2) = \rho^{n-2} G(\xi_1,\xi_2), \quad H_\rho(\bar \xi_1,\bar \xi_2) = \rho^{n-2} H(\xi_1,\xi_2),
\]
and then 
\[
\varphi_\rho (\bar \xi_1,\bar \xi_2) = \rho^{n-2} \varphi (\xi_1,\xi_2),
\]
where
\[
\varphi_\rho (\bar \xi_1,\bar \xi_2) = H_\rho^{1/2}(\bar \xi_1,\bar \xi_1) H_\rho^{1/2}(\bar \xi_2,\bar \xi_2) - G_\rho (\bar \xi_1,\bar \xi_2).
\]

We denote by $\bar \xi_1$ the point $\bar \xi_{1\rho} = (0_{\R^{n-1}},1)$. Arguing as in the proof of Lemma~\ref{flow_closed}, we deduce that $|\bar \xi_1 - \bar \xi_{2\rho}| = O(1)$ as $\rho \to 0$. Then we can assume that $\bar \xi_{2\rho} \to \bar \xi_2$ as $\rho \to 0$ for certain $\bar \xi_2 = (\bar \xi_2',\bar \xi_2^{n})$, where $\xi_2^{n} \ge 1$. Observe that as $\rho \to 0$ the domain $\Omega^\rho$ becomes the half-space $\R^n_+ = \{ \xi = (\xi_1,\dotsc,\xi_n) \in \R^n : \xi_n > 0 \}$. Thereby, following the proof of Lemma~\ref{robin_function_lemma_restricted}, we deduce that $G_\rho$ and $H_\rho$ converge, respectively, to the Green function of the half-space $\R^n_+$, denoted by $ G_+ $, and its regular part $ H_+ $. Likewise, $\varphi_\rho$ and its gradient converge to $\varphi_+$ and its gradient, respectively, where 
\[
\varphi_+ (\xi_1,\xi_2) = H_+^{1/2}(\xi_1,\xi_1) H_+^{1/2}(\xi_2,\xi_2) - G_+ (\xi_1,\xi_2).
\]

The Green function of the half-space can be explicitly written as  (see for instance \cite{Ku1997})
\begin{equation} \label{G+}
  G_+(\xi_1,\xi_2) = \frac{a_{n,s}}{r^{(n-2s)/2}} [1 - d_{n,s} K(r,t)], \quad \xi_i = (\xi_i^1,\dotsc,\xi_i^n) \in \R^n_+,\ i=1,2.
\end{equation}
where
\[
  K(r,t) = \frac{1}{(r+t)^{(n-2)/2}} \int_{0}^{\frac{r}{t}} \frac{(r-tb)^{(n-2)/2}}{b^s(1+b)} \, \di b,
\]
$r = | \xi_1 - \xi_2 |^2$, $t = 4 \xi_1^n \xi_2^n$ and $ d_{n,s} $ is a positive constant. For the simplicity of notation, we shall omit $n$ and $s$ in the constants. Then, after a simple change of variables, the regular part of $ G_+ $ is
\[
  H_+(\xi_1,\xi_2) = \frac{a d}{r^{(n-2s)/2}} K(r,t) = \frac{ad}{t^{1-s}(r+t)^{(n-2)/2}} \int_{0}^{1} \frac{(1-b)^{(n-2)/2}}{b^s(1+\frac{r}{t}b)} \, \di b.
\]
Therefore, the Robin's function associated to the half-space can be written as
\begin{equation} \label{R+}
  R_+(\xi) = H_+(\xi,\xi) = \frac{ad\iota}{2^{n-2s}(\xi^n)^{n-2s}}, \quad \xi = (\xi^1,\dotsc,\xi^n),
\end{equation}
where 
\[
  \iota = \iota_{n,s} = \int_{0}^{1} \frac{(1-b)^{(n-2)/2}}{b^s} \, \di b.
\] 

Now, since $\varphi_\rho (\bar \xi_1,\bar \xi_{2\rho}) = c \rho^{n-2}$, we have
\begin{equation} \label{varphi_xibar_restricted}
	\varphi_+(\bar \xi_1,\bar \xi_2) = 0.
\end{equation}
Assume first that $\bar \xi_2' \neq 0$. Using the above expressions, we see that 
\[
 \varphi_+(\xi_1,\xi_2) = \frac{ad\iota}{2^{n-2s}(\xi_1^n)^{(n-2s)/2} (\xi_2^n)^{(n-2s)/2}} - \frac{a}{r^{(n-2s)/2}} + \frac{ad}{r^{(n-2s)/2}} K(r,t).
\]
Therefore, after differentiating along $ \tau = (0_{\R^n},\xi_2',0) $ and then evaluating at the corresponding points, we deduce that
\begin{align*}
 \nabla \varphi_+(\bar \xi_1,\bar \xi_2) \cdot \tau =&\ \frac{\partial}{\partial r}\left( - \frac{a}{r^{(n-2s)/2}} + \frac{ad}{r^{(n-2s)/2}} K(r,t)\right) (\nabla r \cdot \tau) \\
             =&\ \left[  \frac{(n-2s)}{r} \left( \frac{a}{r^{(n-2s)/2}} - \frac{ad}{r^{(n-2s)/2}} K(r,t) \right) + \frac{2 ad}{r^{(n-2s)/2}} \frac{\partial K}{\partial r}(r,t) \right]|\bar \xi_2'|^2 \\
             =&\ \left( \frac{ (n-2s) ad\iota }{2^{n-2s} (\bar \xi_2^n)^{(n-2s)/2}r} + \frac{2 ad}{r^{(n-2s)/2}} \frac{\partial K}{\partial r}(r,t) \right)|\bar \xi_2'|^2,
\end{align*}
where the last equality is a consequence of \eqref{varphi_xibar_restricted}. Thanks to \eqref{K_r} below, the right-hand side in the previous chain of equalities is positive.  Observe that for $\rho$ sufficiently small, $\tau = (0_{\R^n},\xi_2',0)$ is tangent to $\partial (\Omega_\rho \times \Omega_\rho)$ in $(\xi_{1\rho}, \xi_{2\rho})$. 

Assume now that $\bar \xi_2' = 0$, and suppose that $\bar \xi_2 = (0,\theta_0)$ with $\theta_0 > 1$. Consider the function $\varphi_+(\theta) = \varphi_+(\bar \xi_1,0_{\R^{n-1}},\theta) = \varphi_+(0_{\R^{n-1}},1,0_{\R^{n-1}},\theta)$; let us prove that $\varphi'_+(\theta_0) > 0$. Indeed, from \eqref{G+} and \eqref{R+}, 
\begin{align*}
  \varphi_+(\theta) = \frac{ad\iota}{2^{n-2s} \theta^{(n-2s)/2}} - \frac{a}{(\theta-1)^{n-2s}} + \frac{ad}{(\theta-1)^{n-2s}} K(\theta),
\end{align*}
where 
\[
  K(\theta) = \frac{1}{[(\theta-1)^2+4\theta]^{(n-2)/2}} \int_{0}^{\frac{(\theta-1)^2}{4\theta}} \frac{[(\theta-1)^2-4\theta b]^{(n-2)/2}}{b^s(1+b)} \, \di b.
\]
By differentiating $\varphi_+$ and evaluating at $\theta_0$, one has
\begin{align*}
  \varphi_+'(\theta_0) =&\ (n-2s) \left( - \frac{ad\iota}{2^{n-2s+1} \theta_0^{(n-2s+2)/2}} + \frac{a}{(\theta_0-1)^{n-2s+1}} - \frac{ad}{(\theta_0-1)^{n-2s+1}} K(\theta_0) \right)  \\
          &+ \frac{ad}{(\theta_0-1)^{n-2s}} K'(\theta_0) \\
         =&\ (n-2s) \left( - \frac{ad\iota}{2^{n-2s+1} \theta_0^{(n-2s+2)/2}} + \frac{ad\iota}{2^{n-2s} \theta_0^{(n-2s)/2}(\theta_0-1)}  \right) + \frac{ad}{(\theta_0-1)^{n-2s}} K'(\theta_0)   
\end{align*}
where the last equality is a consequence of \eqref{varphi_xibar_restricted}. Therefore
\[
 \varphi_+'(\theta_0) =  \frac{(n-2s) ad\iota (\theta_0+1)}{2^{n-2s+1} \theta_0^{(n-2s+2)/2}(\theta_0-1)} + \frac{ad}{(\theta_0-1)^{n-2s}} K'(\theta_0).
\]
It is then sufficient to prove that $K'(\theta_0) > 0$. Indeed, it is straightforward to show that
\begin{gather}
\frac{\partial K}{\partial r}(r,t) = \frac{(n-2)t}{2(r+t)^{n/2}} \int_{0}^{\frac{r}{t}} \frac{(r-tb)^{(n-4)/2}}{b^s} \, \di b, \label{K_r}\\
\frac{\partial K}{\partial t}(r,t) = -\frac{(n-2)r}{2(r+t)^{n/2}} \int_{0}^{\frac{r}{t}} \frac{(r-tb)^{(n-4)/2}}{b^s} \, \di b.
\end{gather} 
Then, for all $\theta>1$, we have
\[
  K'(\theta) = \frac{\partial K}{\partial r}(r,t)\ r'(\theta) + \frac{\partial K}{\partial t}(r,t)\ t'(\theta),
\]
with $r(\theta) = (\theta-1)^2$ and $t(\theta) = 4\theta$. Therefore
\begin{align*}
  K'(\theta) =&\ \frac{4(n-2)\theta(\theta-1)}{(\theta+1)^n} \int_{0}^{\frac{(\theta-1)^2}{4\theta}} \frac{[(\theta-1)^2-4\theta b]^{(n-4)/2}}{b^s} \, \di b \\ 
  &- \frac{2(n-2)(\theta-1)^2}{(\theta+1)^n} \int_{0}^{\frac{(\theta-1)^2}{4\theta}} \frac{[(\theta-1)^2-4\theta b]^{(n-4)/2}}{b^s} \, \di b \\
        =&\ \frac{2(n-2)(\theta-1)}{(\theta+1)^{n-1}} \int_{0}^{\frac{(\theta-1)^2}{4\theta}} \frac{[(\theta-1)^2-4\theta b]^{(n-4)/2}}{b^s} \, \di b > 0.
\end{align*}

And then we find that 
\[
  \nabla \varphi_+(\bar \xi_1,\bar \xi_2) \cdot \tau > 0,
\]
where $\tau = (0_{\R^n},0_{\R^{n-1}}, 1)$. Observe that $\tau$ is tangent to $\partial (\Omega_\rho \times \Omega_\rho)$ in $(\xi_{1\rho}, \xi_{2\rho})$, and the proof is complete.

\end{proof}

With the previous results on hand, the proof of Theorem~\ref{two-bubble_solutions} follows exactly as that in the end of the previous section. The details are left to the reader.


\subsection{The subcritical case, one-bubble solutions}

In this subsection we prove Theorem~\ref{theorem2} for the restricted fractional Laplacian. Let us then suppose that the exponent of the nonlinearity in \eqref{fractional_problem_restricted} is $p^*-\ep$ and that $ m=1 $, that is, we consider the subcritical case and study the concentration phenomena for just one bubble. In this case the function $\Psi$ in \eqref{psi_two-bubble_restricted} takes the form
\[
  \Psi(\xi, \Lambda)= \frac{1}{2} H(\xi,\xi) \Lambda^2 - \log \Lambda, \quad \xi \in \Omega, \Lambda>0.
\]
As in the previous section, the Robin's function $ R(\xi) = H(\xi,\xi)$ blows up at the boundary, implying that  its absolute minimums are somehow stable under small variations of it.

\begin{lemma} \label{robin_function_lemma_restricted}
Given $\xi \in \Omega$, we define the function $d(\xi) := \dist(\xi, \partial \Omega)$. Then, there exists positive constants $c_1$ and $c_2$ such that,
\begin{equation} \label{blow-up_restricted}
  c_1 d(\xi)^{2s-n} \le R(\xi) \le c_2 d(\xi)^{2s-n} \quad \text{for all } \xi \in \Omega.
\end{equation}
\end{lemma}

\begin{proof}
Let $\xi_0=(\xi_0^1,\dotsc,\xi_0^n) \in \partial \Omega$, and consider the ball $B := B_{1}(0)$. After a rearrange of variables, we can assume that $ \xi_0 = (1,\dotsc,0) $ and $B \subset \Omega^c$. We shall use the Green function of $B$ and the Kelvin transform to bound from above the Green function of $\Omega$, which we denote by $G$.

The Green function of the unit ball $B$ can be explicitly written as  (see for instance \cite{Ku1997})
\begin{equation} \label{G+_ball}
  G_B(\xi_1,\xi_2) = \frac{a_{n,s}}{r^{(n-2s)/2}} [1 - d_{n,s} K(r,t)], \quad \xi_i = (\xi_i^1,\dotsc,\xi_i^n) \in \R^n_+,\ i=1,2.
\end{equation}
where
\begin{equation} \label{K}
\begin{aligned}
  K(r,t) &= \frac{1}{(r+t)^{(n-2)/2}} \int_{0}^{\frac{r}{t}} \frac{(r-tb)^{(n-2)/2}}{b^s(1+b)} \, \di b \\
         &= \frac{r^{(n-2s)/2}}{t^{1-s}(r+t)^{(n-2)/2}} \int_{0}^{1} \frac{(1-b)^{(n-2)/2}}{b^s(1+\frac{r}{t}b)} \, \di b ,
\end{aligned}
\end{equation}
$r = | \xi_1 - \xi_2 |^2$ and $ t = (1-|\xi_1|^2)(1-|\xi_2|^2) $. Comparing with the Green function $G_+$ in the half-space, see \eqref{G+}, the expression for $G_B$ in terms of $r$ and $t$ is the same. However, $t$ is differently defined here. 

Let us consider the Kelvin transform of $G_B$, and define the function
\begin{align*}
  G_{B^c}(\xi_1,\xi_2) &= |\xi_1|^{2s-n} |\xi_2|^{2s-n} G_B(\xi_1^*,\xi_2^*) \\
                       &= \varGamma(\xi_1-\xi_2) + \frac{ a_{n,s} d_{n,s} |\xi_1|^{2s-n} |\xi_2|^{2s-n} }{{t^*}^{1-s}(r^*+t^*)^{(n-2)/2}} \int_{0}^{1} \frac{(1-b)^{(n-2)/2}}{b^s(1+\frac{r^*}{t^*}b)} \, \di b,
\end{align*}
where $r^* = | \xi_1^* - \xi_2^* |^2$, $ t^* = (1-|\xi_1^*|^2)(1-|\xi_2^*|^2) $, and the last equality is a consequence of \eqref{inversions} and \eqref{K}. The function $ G_{B^c} $ satisfies
\begin{equation}
  \left\{
    \begin{aligned}
       (-\Delta_{|\Omega})^s G_{B^c}(\cdot,\xi_2) &= \delta_{\xi_2}(\cdot) && \text{in } \Omega, \\
       G_{B^c}(\cdot,\xi_2) &\ge 0 && \text{in } \R^n \setminus \Omega;
    \end{aligned}
  \right.
\end{equation}
and, as a consequence of the maximum principle, we deduce that 
\[
  G(\xi_1,\xi_2) \le G_{B^c}(\xi_1,\xi_2) \quad \text{for all } \xi_1,\xi_2 \in \Omega.
\] 

Then, we have that 
\[
  H(\xi_1,\xi_2) \ge  \frac{ a_{n,s} d_{n,s} |\xi_1|^{2s-n} |\xi_2|^{2s-n} }{{t^*}^{1-s}(r^*+t^*)^{(n-2)/2}} \int_{0}^{1} \frac{(1-b)^{(n-2)/2}}{b^s(1+\frac{r^*}{t^*}b)} \, \di b \quad \text{for all } \xi_1,\xi_2 \in \Omega.
\]
Thus, there exists a positive constant $c_1>0$ such that for all $\xi \in \Omega$ close to $\xi_0$,
\begin{align*}
  R(\xi) = H(\xi,\xi) &\ge \frac{a_{n,s} d_{n,s}}{(|\xi|+1)^{n-2s}(|\xi|-1)^{n-2s}} \int_{0}^{1} \frac{(1-b)^{(n-2)/2}}{b^s} \, \di b \\
                      &\ge \frac{c_1}{|\xi - \xi_0|^{n-2s}} 
\end{align*}
(observe that in this case $r^* = 0$ and $t^* = (1-|\xi|^2)^2$). Therefore, taking into account that $\xi_0 \in \partial \Omega$ is arbitrary, we conclude that in a neighborhood of $\partial \Omega$ there exist a constant $c_1 > 0$ such that $R(\xi) \ge c_1 d(\xi)^{2s-n}$. The smoothness of $H$ in $ \Omega $ allows us to extend this inequality to the whole domain. 

The other inequality in \eqref{blow-up_restricted} can be proven by a similar argument using an interior ball instead. The details are left to the reader. 
\end{proof}

\begin{proof}[Proof of Theorem~\ref{theorem2}: the restricted fractional Laplacian case]
	Thanks to the previous lemma, there still exist absolute minimums of small perturbations of $ R(\xi) = H(\xi,\xi)$. Theorem~\ref{theorem2} is then a consequence of this fact together Lemma~\ref{reduction_restricted} and \eqref{reduced_energy_gradient_restricted}.
\end{proof}



\end{document}